%% file: Hierarchy-1387-10-8-arxiv.tex
\documentclass[11pt]{article}
\usepackage[usenames]{color}

\usepackage{gastex}
\usepackage{gastex}
\usepackage{graphicx,amsmath,amssymb}
\input{macs}
\def\emline#1#2#3#4#5#6{%
\put(#1,#2){\special{em:moveto}}%
\put(#4,#5){\special{em:lineto}}}
\newtheorem{precor}{{\bf Corollary}}

\newenvironment{cor}{\begin{precor}{\hspace{-0.5
em}{\bf.\ }}}{\end{precor}}
\newtheorem{precon}{{\bf Conjecture}}

\newenvironment{con}{\begin{precon}{\hspace{-0.5
em}{\bf.\ }}}{\end{precon}}
\newtheorem{predefin}{{\bf Definition}}

\newenvironment{defin}[1]{\begin{predefin}{\hspace{-0.5
em}{\bf.\ }}{\rm
#1}\hfill{$\spadesuit$}}{\end{predefin}}
\newtheorem{preexm}{{\bf Example}}

\newenvironment{exm}[1]{\begin{preexm}{\hspace{-0.5
em}{\bf.\ }}{\rm #1}\hfill{$\clubsuit$}}{\end{preexm}}

\newtheorem{preappl}{{\bf Application}}

\newtheorem{prelem}{{\bf Lemma}}

\newtheorem{preproof}{{\bf Proof.\ }}

\newenvironment{proof}[1]{\begin{preproof}{\rm
#1}\hfill{$\blacksquare$}}{\end{preproof}}
\newtheorem{presproof}{{\bf Sketch of Proof.\ }}

\newtheorem{prethm}{{\bf Theorem}}

\newenvironment{thm}{\begin{prethm}{\hspace{-0.5
em}{\bf.\ }}}{\end{prethm}}
\newtheorem{prealphthm}{{\bf Theorem}}

\newenvironment{alphthm}{\begin{prealphthm}{\hspace{-0.5
em}{\bf.\ }}}{\end{prealphthm}}
\newtheorem{prepro}{{\bf Proposition}}

\newenvironment{pro}{\begin{prepro}{\hspace{-0.5
em}{\bf.\ }}}{\end{prepro}}
\newtheorem{preprb}{{\bf Problem}}

\newenvironment{prb}{\begin{preprb}{\hspace{-0.5
em}{\bf.\ }}}{\end{preprb}}
\def\sds {{\sf sds }}
\def\Ssdn {{\sf Ssdn }}
\def\Wsdn {{\sf Wsdn }}
\def\Ssds {{\sf Smsds }}
\def\Wsds {{\sf Wmsds }}

\def\conct[#1,#2]{\mbox {${#1} \leftrightarrow {#2}$}}
\def\dconct[#1,#2]{\mbox {${#1} \rightarrow {#2}$}}
\def\deg[#1,#2]{\mbox {$d_{_{#1}}(#2)$}}
\def\mindeg[#1]{\mbox {$\delta_{_{#1}}$}}
\def\maxdeg[#1]{\mbox {$\Delta_{_{#1}}$}}
\def\outdeg[#1,#2]{\mbox {$d_{_{#1}}^{^+}(#2)$}}
\def\minoutdeg[#1]{\mbox {$\delta_{_{#1}}^{^+}$}}
\def\maxoutdeg[#1]{\mbox {$\Delta_{_{#1}}^{^+}$}}
\def\indeg[#1,#2]{\mbox {$d_{_{#1}}^{^-}(#2)$}}
\def\minindeg[#1]{\mbox {$\delta_{_{#1}}^{^-}$}}
\def\maxindeg[#1]{\mbox {$\Delta_{_{#1}}^{^-}$}}
\def\isdef{\mbox {$\ \stackrel{\rm def}{=} \ $}}
\def\dre[#1,#2,#3]{\mbox {${\cal E}_{_{#3}}(#1,#2)$}}
\def\pdre[#1,#2,#3]{\mbox {${\cal P}_{_{#3}}(#1,#2)$}}
\def\var[#1,#2]{\mbox {${\rm Var}_{_{#1}}(#2)$}}
\def\ls[#1]{\mbox {$\xi^{^{#1}}$}}
\def\hom[#1,#2]{\mbox {${\rm Hom}({#1},{#2})$}}
\def\onvhom[#1,#2]{\mbox {${\rm Hom^{v}}(#1,#2)$}}
\def\onehom[#1,#2]{\mbox {${\rm Hom^{e}}(#1,#2)$}}
\def\core[#1]{\mbox {$#1^{^{\bullet}}$}}
\def\cay[#1,#2]{\mbox {${\rm Cay}({#1},{#2})$}}
\def\cays[#1,#2]{\mbox {${\rm Cay_{s}}({#1},{#2})$}}
\def\dirc[#1]{\mbox {$\stackrel{\rightarrow}{C}_{_{#1}}$}}
\def\cycl[#1]{\mbox {${\bf Z}_{_{#1}}$}}

\begin{document}
\footnotetext[1]{Correspondence should be addressed to {\tt
daneshgar@sharif.ir}.}
\begin{center}
{\Large \bf  ON SEQUENTIAL COLORING OF GRAPHS\\ AND ITS DEFINING SETS}\\
\vspace*{0.5cm}
{\bf Amir Daneshgar\\ and \\Roozbeh Ebrahimi Soorchaei}\\
{\it Department of Mathematical Sciences} \\
{\it Sharif University of Technology} \\
{\it P.O. Box {\rm 11155--9415}, Tehran, Iran.}\\
2008, Dec., 28\\ \ \\

\end{center}
\begin{abstract}
\noindent In this paper, based on the contributions of Tucker (1983)
and Seb{\H{o}} (1992), we generalize the concept of a sequential
coloring of a graph to a framework in which the algorithm may use a
coloring rule-base obtained from suitable forcing structures. In
this regard, we introduce the {\it weak} and {\it strong sequential
defining numbers} for such colorings and as the main results, after
proving some basic properties, we show that these two parameters are
intrinsically different and their spectra are nontrivial. Also, we
consider the natural problems related to the complexity of computing
such parameters and we show that in a variety of cases these
problems are ${\bf NP}$-complete. We conjecture that this result
does not depend on the rule-base for all nontrivial cases.
\end{abstract}
\section{Introduction}\label{INTRO}
Graph coloring, as a special case of the graph homomorphisms
problem, is widely known to be among the most fundamental problems
in graph theory and combinatorics. On the other hand, a long list of
applications of the problem within and out of graph theory itself,
not only adds to the importance of the problem, but also, asks for
efficient and easy-to-use solutions in different cases (e.g. see
\cite{DW90,HN04,JT95}). Unfortunately, based on some recent deep
nonapproximability results (e.g. among other references see
\cite{AR?,GHS02} for the applications of the PCP theorem), it is
known that the coloring problem is among the hardest problems in
${\bf NP}$ and does not admit an efficient solution/approximation,
unless ${\bf P}={\bf NP}$ (or similar coincidence
for randomized classes as ${\bf RP}$). Therefore, it is quite
important, at least as far as the applications are concerned, to
acquire a good understanding of the coloring structures of graphs,
and to develop efficient algorithms
for the tractable cases.\\
Sequential coloring of a given graph is an alternative approach to
the classical coloring problem that may not lead to the same
solution (e.g. see \cite{HDW89,MAF03,SEBO,TUC83}), however, based on the algorithm and its rule-base, it is
conceivable to seek a set of vertices with preset colors that guides
the algorithm to the ideal solution.
This approach not only can be considered as a partial solution to the original problem, but also will give rise to a hierarchy of graphs,
based on their coloring structures, which is ordered using the rule-base of the algorithm.\\
Probably, the most popular algorithm of this type is the so-called
{\it greedy algorithm} that uses the very simple (and independent of
the structure) rule of setting the first color available, regardless
of the neighborhood structure of the vertex visited. Our approach in
this paper, is to enhance such an algorithm to be able to use a
local coloring rule-base, depending on the neighborhood structure of
the vertices, and use these rule-bases to introduce a hierarchy
of graphs based on their coloring structures. \\
To the best of authors' knowledge, {\it sequential coloring} of
graphs is formally first introduced by A.~Tucker in \cite{TUC83} to
study perfect graphs and is also studied afterwards (Particularly,
see \cite{SEBO} for a very interesting application and connections
to other classes of graphs as perfect and uniquely colorable
graphs). The very closely related concept of {\it online coloring}
of graphs in which one assumes that the vertices are ordered (say
$\{v_{_{1}},\cdots,v_{_{n}}\}$) and the algorithm at the $k$'th
vertex $v_{_{k}}$ is only informed about the data of vertices
$\{v_{_{1}},\cdots,v_{_{k-1}}\}$, is also extensively studied for
its important applications in real industrial problems, where the
greedy version is sometimes called the {\it first fit} or {\it
Grundy} coloring. However, in sequential coloring of graphs, we
assume that the algorithm at each vertex is informed about the
structure and the data of a neighborhood of the vertex (regardless
of the indices). Therefore, from an algorithmic point of view, the
main difference between the online and the sequential approaches
falls in the way that the algorithm is informed about the structure
of the graph. In this article we concentrate on the concept of
sequential coloring of graphs, while most of
the results can be extended to the online version applying the necessary modifications.\\
In our approach, we generalize the sequential model in two ways.
Firstly, we let the algorithm have access to a neighborhood of
radius $d$ of each vertex, where $d$ is a constant of the algorithm,
and  secondly, we confine the coloring rules of the algorithm to be
only accessible through a fixed rule-base ${\cal R}$. Such
rule-bases are introduced by using the concept of a {\it forcing
structure} discussed in
Section~\ref{FORCING} (see \cite{DAN97-2,DAN97-3,DAN98-1,DAN99-1,DANA00,DAN01} for the background).
Our results in this section show that there is a general machinery to generate a variety of rule-bases
to match the user's needs. \\
Our main objective in this paper is to study the defining sets of
sequential colorings of graphs and analyze the computational
complexity of the related problems. As a brief historical note we
may note that the concept of a defining set of graph colorings is
extensively studied (e.g. see \cite{DMRS03}), specially for the case
of Cartesian product of complete graphs which is essentially what is
already known as {critical sets} of Latin squares (e.g. see
\cite{CAV,KEED04,KEED96}). To the best of the first author's
knowledge, the concept of an online critical set for a Latin square
is first introduced by E.~Mendelsohn (the first author heard about
this from E.~Mendelsohn during meetings and talks at Sharif
University of Technology on November $1998$). The concept was soon
generalized to the concept of a defining set of an online graph
coloring by M.~Zaker and the first
results on these are published by him (see \cite{VAN,ZAK01-1,ZAK01-2,ZAK06,ZAK08-1}).\\
As far as we are aware, the concept of a defining set of sequential
colorings of a graph is first introduced by the first author and it
is appeared in the main result of \cite{DAHT?}, where it is proved
that any boolean function can be simulated by a graph, where to
compute
the function for some given values, one may use the values as colors of a defining set of a sequential coloring of the corresponding graph.\\
Needless to say, the contributions of A.~Tucker and A.~Seb{\H{o}}
ought to be considered as the origins of sequential colorings of
graphs, however the approach in \cite{DAHT?} and in this paper is
mainly inspired by the applications of the concept of {\it effective}
defining sets in cryptography (also see \cite{ZAK08-2}). Therefore,
in the rest of this section we present a short overview of the
contributions of A.~Tucker and A.~Seb{\H{o}} and we also provide the
basic concepts we need in the rest of the article, namely {\it
forcing structures}
and {\it graph amalgams}.\\
In Section~\ref{SEQDEF} we precisely define the concept of a {\it
sequential coloring} of a graph, as well as the spectra of {\it
weak} and {\it strong defining numbers}, and moreover, we prove some
basic results in this regard. Specially, among our main results in
this section we show that the two concepts of {\it weak} and {\it
strong} defining numbers are intrinsically different
(Theorem~\ref{WSDIF}) and also we show that each
spectrum is completely nontrivial (Theorem~\ref{NONTRIVSPEC}).\\
Our main contribution in Section~\ref{COMPLEXITY} are a couple of
${\bf NP}$-completeness results (generally) stating that computation
of weak and strong defining numbers are {${\bf NP}$-complete}
problems. Although we have not been able to obtain efficient and
easy to present proofs for the fact that the ${\bf NP}$-completeness
result does not depend on the rule-base {\it in general}, we
strongly believe that this is the case for almost all nontrivial
cases, where Theorems~\ref{T3}, \ref{T4} and \ref{T5} as the main
results
of this section are positive justifications in this regard.\\
Throughout the paper we will be considering finite simple graphs.
The vertex set of a graph $\matr{G}$ is denoted by $V(\matr{G})$ and
the edge set by $E(\matr{G})$, and for any subset $A \subseteq
V(\matr{G})$ the notation $\matr{G}[[A]]$ stands for the vertex
induced subgraph on $A$. An edge with two end vertices $u$ and $v$
is denoted by $uv$ and $N_{_{\matr{G}}}(v) \isdef \{u \in
V(\matr{G}) \ \ | \ \ uv \in E(G) \}$ is the set of neighbours of
the vertex  $v$ in $\matr{G}$. Throughout the paper, $|A|$ denotes
the size of the set $A$. Also, $\matr{K}_{_{n}}$ is the complete
graph on $n$ vertices and $cl(\matr{G})$ is the clique number of the
graph $\matr{G}$. A graph $\matr{G}$ is said to be {\it
$t$-colorable} if it admits a proper vertex $t$-coloring. Also,
$\matr{G}$ is said to be $t$-chromatic if $t=\chi(\matr{G})$, where
$\chi(\matr{G})$ is the ordinary chromatic number of $\matr{G}$. For
basic concepts of graph theory we refer to \cite{WES96}.\\
In what follows we recall a more or less standard notation for
amalgams which is adopted from \cite{DAHT?}
(the interested reader may also consult \cite{DAHT?} for more on this or read the Appendix of this article).\\
Let $X=\{x_{_{1}},x_{_{2}},\ldots,x_{_{k}}\}$ and $\matr{G}$ be a
set and a graph, respectively, and also, consider a one-to-one map
$\varrho: X \longrightarrow V(\matr{G})$. Evidently, one can
consider $\varrho$ as a graph monomorphism from the empty graph
$\matr{X}$ on the vertex set $X$ to the graph $\matr{G}$, where in
this setting we interpret the situation as a {\it labeling} of some
vertices of $\matr{G}$ by the elements of $X$. The data introduced
by $(X,\matr{G},\varrho)$ is called a {\it marked graph} $\matr{G}$
marked by the set $X$ through the map $\varrho$. Note that (by abuse
of language) we may introduce the corresponding marked graph as
$\matr{G}[x_{_{1}},x_{_{2}},\ldots,x_{_{k}}]$ when the definition of
$\varrho$ (especially its range) is clear from the context. Also,
(by abuse of language) we may refer to {\it the vertex $x_{_{i}}$}
as the vertex $\varrho(x_{_{i}}) \in V(\matr{G})$. This is most
natural when $X \subseteq V(G)$ and vertices in $X$ are marked by
the corresponding elements in
$V(G)$ through the identity mapping.
We use the following notation
$$\matr{G}[x_{_{1}},x_{_{2}},\ldots,x_{_{k}}]+\matr{H}[y_{_{1}},y_{_{2}},\ldots,y_{_{l}}]$$
for the graph (informally) constructed by taking the disjoint union
of the marked graphs $\matr{G}[x_{_{1}},x_{_{2}},\ldots,x_{_{k}}]$
and $\matr{H}[y_{_{1}},y_{_{2}},\ldots,y_{_{l}}]$ and then {\it
identifying} the vertices with the same mark (for a formal and
precise definition see the Appendix). It is understood that a
repeated appearance of a graph structure in an expression as
$\matr{G}[v]+\matr{G}[v,w]$ is always considered as different
isomorphic copies of the structure marked properly by the indicated
labels (e.g. $\matr{G}[v]+\matr{G}[v,w]$ is an amalgam constructed
by two different isomorphic copies of $\matr{G}$ identified on the
vertex $v$ where the
vertex $w$ in one of these copies is marked).\\
By $\matr{K}_{_{k}}[v_{_{1}},v_{_{2}},\ldots,v_{_{k}}]$ we mean a
$k$-clique on $\{v_{_{1}},v_{_{2}},\ldots,v_{_{k}}\}$ marked by its
own set of vertices. Specially, a single edge is denoted by
$\matr{\varepsilon}[v_{_{1}},v_{_{2}}]$ (i.e.,
$\matr{\varepsilon}[v_{_{1}},v_{_{2}}] =
\matr{K}_{_{2}}[v_{_{1}},v_{_{2}}]$). As one more simple example
note that
$\matr{\varepsilon}[v_{_{1}},v_{_{2}}]+\matr{\varepsilon}[v_{_{2}},v_{_{3}}]$
is a path of length $2$ on the vertex set
$\{v_{_{1}},v_{_{2}},v_{_{3}}\}$.
\subsection{Forcing in graph coloring}\label{FORCING}
Given a graph $\matr{G}$, consider a set ${\cal L}=\{L_{_{v}}\}_{_{v
\in V(\matr{G}) }}$ with $L_{_{v}} \subseteq \{1,\ldots,t\}$ such
that $t \geq \chi(\matr{G}) $ is a fixed integer. Hereafter,
$$\|{\cal L}\| \isdef \displaystyle{\sum_{v \in V(\matr{G})}}\ |L_{_{v}}|.$$
The list coloring problem, $(\matr{G},{\cal L},t)$, is to find a  coloring
$\sigma:V(\matr{G}) \longrightarrow \{1,\ldots,t\}$ such that for all pairs of distinct vertices
$u$ and $v$,
$$(\sigma(v) \in L_{_{v}}) \quad {\rm and} \quad (uv \in E(G) \Rightarrow \sigma(u) \not = \sigma(v)).$$
To see that this can also be viewed as an embedding problem,
assume that the list coloring problem $(\matr{G},{\cal L},t)$  is given,
and let $\matr{H}=\matr{G} \cup \matr{K}_{_{t}}$ with $V(\matr{K}_{_{t}})=\{v_{_{1}}, \ldots
,v_{_{t}}\}$. Also, assume that (without loss of generality) we
have fixed the colors  of vertices $V(\matr{K}_{_{t}})=\{v_{_{1}},
\ldots ,v_{_{t}}\}$ such that $v_{_{i}}$ has taken the color $i$
for all $i \in \{1, \ldots ,t\}$. Now, one may construct a new
graph, $\  \tilde{\matr{H}}_{_{\matr{G},\Phi,t}}$, such that for each $u \in V(\matr{G})$
and $L_{_{u}} \in {\cal L}$ we add new edges
$\Phi_{_{u}}=\{uv_{_{i}}\}_{_{i \not \in L_{_{u}}}}$ to the graph
$\matr{H}$. If
$$\Phi \isdef \displaystyle{\bigcup_{_{u \in V(\matr{G})}}} \Phi_{_{u}},$$
then it is clear that there is a one-to-one correspondence between the set of $t$-colorings of the graph $\tilde{\matr{H}}_{_{\matr{G},\Phi,t}}$
and the set of solutions ($t$-colorings) of the list coloring problem $(\matr{G},{\cal L},t)$ (given by restriction).
A graph $\matr{G}$ is called uniquely
${\cal L}$-list-colorable, if the list coloring problem,
$(\matr{G},{\cal L},t)$, on $\matr{G}$ with lists ${\cal L}=\{L_{_{v}}\}_{_{v
\in V(\matr{G}) }}$ has a unique solution.
Also, when the lists are all equal to the set $\{1,\ldots,t\}$,
such a graph is called a uniquely $t$-colorable graph or a $t$-UCG for short if the list coloring problem has a unique solution up to permutation
of colors which means that the color classes are fixed for all solutions. For more on this subject see \cite{DAHA05,JT95}.\\
Hereafter,  we will freely switch between the embedding
setup and the list coloring approach. Also, we write ${\cal L}
\subseteq {\cal L}'$ as a shorthand for $L_{_{v}} \subseteq L'_{_{v}}$
for all $v$.\\
Consider a list coloring problem $(\matr{H},{\cal L},t)$, and a subgraph $\matr{G} \leq \matr{H}$
such that
$$A \isdef \displaystyle{\bigcup_{_{u \in V(\matr{G})}}} L_{_{u}}, \quad {\rm and} \quad |A|=k \leq t,$$
where the list coloring problem $(\matr{G},{\cal L},k)$ has a unique solution (up to permutation of colors).
Then for any vertex $v \not \in V(\matr{G})$ whose neighborhood in $\matr{G}$ intersects all
color-classes of $(\matr{G},{\cal L},k)$, one may replace $L_{_{v}}$ by
$L_{_{v}}-A$ without any change in the solution-set of the
list coloring problem $(\matr{H},{\cal L},t)$. Such a situation is called a
{\it type~$1$ forcing} on $v$ by $\matr{G}$ \cite{DAN97-2,DAN97-3,DAN98-1,DAN01}. Note that a special case of this
situation is when $\matr{G} \leq \matr{H}$ is a $k$-UCG.
\begin{exm}{\label{FORC1} Consider the minimal $3$-UCG of Figure~\ref{G3TRU81}, and let's try to find
a $3$-coloring $\sigma$ of it. It is quite easy to see that without loss of generality we have
$$\sigma(v_{_{1}})=\sigma(v_{_{4}})=1, \quad \sigma(v_{_{2}})=2 \quad {\rm and} \quad \sigma(v_{_{3}})=3.$$
Then it is clear that the set of possible colors for both of the vertices $v_{_{5}}$ and $v_{_{6}}$
is $\{2,3\}$, and consequently, by a type~$1$ forcing on the $\matr{K}_{_{2}}$-clique induced on $\{v_{_{5}},v_{_{6}}\}$
we have $\sigma(v_{_{7}})=1$.
}\end{exm}
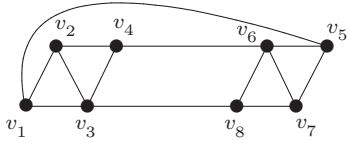
\begin{figure}[t]
\centering
\special{em:linewidth 0.4pt} \unitlength 0.40mm
\linethickness{0.4pt}
\begin{picture}(124.00,76.67)
\put(20.00,20.33){\circle*{4.00}}
\put(40.33,20.00){\circle*{4.00}}
\put(90.00,20.00){\circle*{4.00}}
\put(109.67,20.00){\circle*{4.00}}
\put(30.00,40.00){\circle*{4.00}}
\put(50.00,40.00){\circle*{4.00}}
\put(100.00,40.00){\circle*{4.00}}
\put(120.00,40.00){\circle*{4.00}}
\emline{20.00}{20.33}{1}{30.00}{40.00}{2}
\emline{30.00}{40.00}{3}{50.00}{40.00}{4}
\emline{50.00}{40.00}{5}{40.33}{20.00}{6}
\emline{40.33}{20.00}{7}{30.00}{39.67}{8}
\emline{20.00}{20.00}{9}{40.33}{20.00}{10}
\emline{40.33}{20.00}{11}{90.00}{20.00}{12}
\emline{90.00}{20.00}{13}{109.67}{20.00}{14}
\emline{109.67}{20.00}{15}{120.00}{40.00}{16}
\emline{120.00}{40.00}{17}{100.00}{40.00}{18}
\emline{100.00}{40.00}{19}{90.00}{20.00}{20}
\emline{109.67}{20.00}{21}{100.00}{40.00}{22}
\emline{100.00}{40.00}{23}{50.00}{40.00}{24}
\emline{19.67}{20.00}{25}{19.12}{22.96}{26}
\emline{19.12}{22.96}{27}{18.74}{25.78}{28}
\emline{18.74}{25.78}{29}{18.55}{28.48}{30}
\emline{18.55}{28.48}{31}{18.52}{31.04}{32}
\emline{18.52}{31.04}{33}{18.68}{33.46}{34}
\emline{18.68}{33.46}{35}{19.01}{35.76}{36}
\emline{19.01}{35.76}{37}{19.53}{37.92}{38}
\emline{19.53}{37.92}{39}{20.21}{39.94}{40}
\emline{20.21}{39.94}{41}{21.08}{41.84}{42}
\emline{21.08}{41.84}{43}{22.12}{43.60}{44}
\emline{22.12}{43.60}{45}{23.34}{45.23}{46}
\emline{23.34}{45.23}{47}{24.74}{46.72}{48}
\emline{24.74}{46.72}{49}{26.31}{48.09}{50}
\emline{26.31}{48.09}{51}{28.06}{49.32}{52}
\emline{28.06}{49.32}{53}{29.99}{50.41}{54}
\emline{29.99}{50.41}{55}{32.09}{51.38}{56}
\emline{32.09}{51.38}{57}{34.37}{52.21}{58}
\emline{34.37}{52.21}{59}{36.83}{52.91}{60}
\emline{36.83}{52.91}{61}{39.47}{53.47}{62}
\emline{39.47}{53.47}{63}{42.28}{53.90}{64}
\emline{42.28}{53.90}{65}{45.27}{54.20}{66}
\emline{45.27}{54.20}{67}{48.44}{54.37}{68}
\emline{48.44}{54.37}{69}{51.78}{54.40}{70}
\emline{51.78}{54.40}{71}{55.30}{54.30}{72}
\emline{55.30}{54.30}{73}{59.00}{54.07}{74}
\emline{59.00}{54.07}{75}{62.88}{53.70}{76}
\emline{62.88}{53.70}{77}{66.93}{53.21}{78}
\emline{66.93}{53.21}{79}{71.16}{52.57}{80}
\emline{71.16}{52.57}{81}{75.57}{51.81}{82}
\emline{75.57}{51.81}{83}{80.15}{50.91}{84}
\emline{80.15}{50.91}{85}{84.91}{49.88}{86}
\emline{84.91}{49.88}{87}{89.85}{48.72}{88}
\emline{89.85}{48.72}{89}{94.97}{47.42}{90}
\emline{94.97}{47.42}{91}{100.26}{45.99}{92}
\emline{100.26}{45.99}{93}{105.73}{44.43}{94}
\emline{105.73}{44.43}{95}{111.38}{42.73}{96}
\emline{111.38}{42.73}{97}{120.00}{40.00}{98}
\put(17.33,12.67){\makebox(0,0)[cc]{$_{v_{_{1}}}$}}
\put(33.67,46.00){\makebox(0,0)[cc]{$_{v_{_{2}}}$}}
\put(40.33,12.67){\makebox(0,0)[cc]{$_{v_{_{3}}}$}}
\put(53.00,46.33){\makebox(0,0)[cc]{$_{v_{_{4}}}$}}
\put(113.67,13.00){\makebox(0,0)[cc]{$_{v_{_{7}}}$}}
\put(90.00,12.66){\makebox(0,0)[cc]{$_{v_{_{8}}}$}}
\put(94.00,44.00){\makebox(0,0)[cc]{$_{v_{_{6}}}$}}
\put(124.00,44.67){\makebox(0,0)[cc]{$_{v_{_{5}}}$}}
\end{picture}
\caption{\protect\label{G3TRU81} A minimal $3$-UCG (see Examples~\ref{FORC1} and \ref{FORC2}).}
\end{figure}
Let us recall the following definition and theorem from \cite{DAN99-1}.
For a $(t-1)$--chromatic graph $\matr{G}$, a list (repetition is allowed)
$${\cal F}=\{(i,W_{_{i}}) \ | \ 1 \leq i \leq l\}$$
with $W_{_{i}} \subseteq V(\matr{G})$ is called a
{\it transverse system}
for $\matr{G}$ if either
$$\matr{G}=\matr{K}_{_{t-1}} \quad \& \quad {\cal F}=\{(1,V(\matr{G}))\},$$
or both of the
following conditions are satisfied;
\begin{itemize}
\item{For every $(t-1)$--coloring $\sigma$ of $\matr{G}$,
if $(i,W) \in {\cal F}$ then $W$ has nonempty
intersection with all color--classes of $\sigma$.}
\item{
For every
$t$--coloring $\sigma : V(\matr{G}) \stackrel{\rm onto}{\longrightarrow}
\{ 1, \dots , t \}$ of $\matr{G}$, there exists
$(i,W) \in {\cal F}$ such that $W$
has nonempty intersection with
all color--classes of $\sigma$.}
\end{itemize}
\begin{alphthm}\label{TRANSVERSE}{\rm~\cite{DAN99-1}}
Let $\matr{H}$ be a $t$--chromatic graph
such that in every
$t$--coloring of $\matr{H}$
there is a fixed color--class $V^{^{*}}$ consisting of $m$
specified vertices $v_{_{1}}, \dots , v_{_{m}}, \ \ (m \geq 1)$;
and consider $\matr{G}=\matr{H}-V^{^{*}}$.
Also define
${\cal F}=\{ (i,N_{_{\matr{G}}}(v_{_{i}}))
\ \ | \ \  v_{_{i}} \in V^{^{*}} \}$. Then,
\begin{itemize}
\item[a {\rm )}]{$\chi(\matr{G})=t-1$ and ${\cal F}$ is a transverse system
for $\matr{G}$.}
\end{itemize}
Moreover if $cl(\matr{H}) \leq t-1$ then,
\begin{itemize}
\item[b {\rm )}]{$cl(\matr{G}[[W]]) \leq t-2$
for every $(i,W) \in {\cal F}$,
and $cl(\matr{G}) \leq t-1$.}
\end{itemize}
Conversely, let $\matr{G}$ be a $(t-1)$--chromatic graph
and let  ${\cal F} \subseteq {\bf N} \times 2^{V(\matr{G})}$
be a transverse system for $\matr{G}$.
Then the graph $\matr{H}$ obtained by adding to $\matr{G}$ new
vertices $v_{_{i}}$, for each $((i,W_{_{i}}) \in {\cal F})$,
and joining each $v_{_{i}}$
to all vertices in $W_{_{i}}$ is a $t$--chromatic graph
such that in any one of its $t$--colorings the class
$V^{^{*}}=\{v_{_{i}}\ \ |
\ \ (i,W_{_{i}}) \in {\cal F} \}$ is fixed.
If in addition {\rm (}$b${\rm )} is also fulfilled then
$cl(\matr{H}) \leq t-1$.
\end{alphthm}
Consider a list coloring problem $(\matr{M},{\cal L},l)$, and subgraphs $\matr{G} \leq \matr{H} \leq \matr{M}$,
with \linebreak
 $V^{^{*}} \isdef V(\matr{H})-V(\matr{G})$, such that $\matr{G}$ and $\matr{H}$ satisfy all conditions of Theorem~\ref{TRANSVERSE}.
Then one may replace the lists of all vertices in $V^{^{*}}$ by
$$L \isdef \displaystyle{\bigcap_{_{u \in V^{^{*}}}}} L_{_{u}},$$
without any change in the solution-set of the
list coloring problem $(\matr{H},{\cal L},l)$. Such a situation is called a
{\it type~$2$ forcing} on $V^{^{*}}$ by $\matr{G}$ \cite{DAN97-2,DAN97-3,DAN98-1,DAN01}. Note that a special case of this
situation is when $\matr{H}$ is a $t$-UCG.
\begin{exm}{\label{FORC2}Again, consider the minimal $3$-UCG of Figure~\ref{G3TRU81}, and let's try to find
a $3$-coloring $\sigma$ of it. It is quite easy to see that without loss of generality we have
$$\sigma(v_{_{1}})=1, \quad \sigma(v_{_{2}})=2 \quad {\rm and} \quad \sigma(v_{_{3}})=3.$$
Then it is clear that the set of possible colors for the vertex $v_{_{5}}$ is $\{2,3\}$ and
the set of possible colors for the vertex $v_{_{8}}$ is $\{1,2\}$.
Now,  by a type~$2$ forcing on the subgraph induced on $\{v_{_{5}},v_{_{6}},v_{_{7}},v_{_{8}}\}$,
in which both $v_{_{5}}$ and $v_{_{8}}$ have the same color in any $3$-coloring,
we may replace the list of colors of these vertices by the intersection, which shows that,
$$\sigma(v_{_{5}})=\sigma(v_{_{8}})=2.$$
}\end{exm}
Hence, using type~$1$ and $2$ forcings along with a suitably chosen uniquely colorable graph one may construct a variety of
forcing structures to analyze or construct new coloring structures.
\subsection{Sequential coloring and perfect graphs} \label{PERFECT}
In \cite{TUC83} Tucker introduces an algorithmic approach to the
Strong Perfect Graph Conjecture.  He calls a $t$--UCG, $\matr{G}$,
{\it sequentially $t$--colorable}, if the $n$ vertices of $\matr{G}$
can be indexed as $v_{_{1}}, \dots ,v_{_{n}}$ such that $v_{_{1}},
\dots ,v_{_{t}}$ are in a clique (and initially are arbitrarily
given different colors) and after $v_{_{1}}, \dots ,v_{_{m-1}} \ \
(m > t)$ are colored, $v_{_{m}}$ has a {\it forced} coloring by one
of the following
sequential coloring rules.\\
\begin{itemize}
\item[1 {\rm )}]{If $v_{_{m}}$ is adjacent to previously colored vertices
of $t-1$ (but not $t$) different colors, then $v_{_{m}}$ must be
given the one color not used by its previously colored neighbors.}
\item[2 {\rm )}]{If $v_{_{j}}$ and $v_{_{m}}$, $(j < m)$, are each adjacent
to all vertices of a clique of size $t-1$ (but $v_{_{j}}$ and
$v_{_{m}}$ are not adjacent to each other), then $v_{_{m}}$ must be
given the same color as $v_{_{j}}$.}
\item[3 {\rm )}]{If all vertices of color $i$ have been sequentially colored
(using rules $1$ and $2$), these vertices can be deleted and $t$ can
be reduced by one in subsequent use of rules $1$ and $2$.}
\end{itemize}
Tucker also calls a $t$--UCG {\it strongly sequentially
$t$--colorable} if it can be sequentially colored starting from any
$t$--clique and using only rule $1$. In these directions he proves
\begin{alphthm}{\rm \cite{TUC83}}\
If $\matr{G}$ is a comparability graph or the complement of a
comparability graph, then $\matr{G}$ is uniquely $t$--colorable if
and only if $\matr{G}$ is sequentially $t$--colorable.
\end{alphthm}
\begin{alphthm}{\rm \cite{TUC83}}\
If $\matr{G}$ is a chordal graph or the complement of a chordal
graph, then $\matr{G}$ is uniquely $t$--colorable if and only if
$\matr{G}$ is sequentially $t$--colorable.
\end{alphthm}
Tucker also  proposes the following conjecture.
\begin{con}{\rm \cite{TUC83}}\
For perfect graphs, unique $t$--colorability is equivalent to
sequential $t$--colorability.
\end{con}
It is known that $\matr{G}-v$ is a uniquely colorable graph for any minimal imperfect graph $\matr{G}$
and any one of its vertices $v$. A well-known approach to the Strong Perfect Graph Conjecture (SPGC) is to study the following two problems:
\begin{itemize}
\item{Prove the existence of Tucker forcing structures in uniquely colorable perfect graphs, and specially in cases where
the graph is obtained from a minimal imperfect graph by excluding one of the vertices.}
\item{Prove that a minimal imperfect graph with enough Tucker forcing structures is either an odd hole or an antihole.}
\end{itemize}
For more on this the reader may refer to \cite{FS90,KS97,SEBO} where
in \cite{SEBO}, among other results, some nice reformulations of
SPGC in relation with Tucker's forcing structures and unique
colorability are given.
\section{Sequential coloring of ordered list-graphs}\label{SEQDEF}
\begin{defin}{
Any marked graph $\varrho:\{1,2,\ldots,|V(\matr{G})|\}
\longrightarrow \matr{G}$ will be called an {\it ordered} graph and
will be denoted in an abbreviated style as $\matr{G}[\varrho]$ or
$\matr{G}[n]$ when $|V(\matr{G})|=n$ and
$\varrho$ is clear from the context (usually assumed to be $i \mapsto v_{_{i}}$). \\
Let $t$ be a fixed integer and ${\cal
L}=\{L_{_{1}},L_{_{2}},\ldots,L_{_{n}}\}$ be a set of lists for
which $L_{_{i}} \subseteq \{1,2,\ldots,t\}$ for each $i$. Then a
graph $\matr{G}$ with $n=|V(\matr{G})|$ marked by ${\cal L}$ is
called a {\it list-graph} and is denoted by $\matr{G}[{\cal L}]$
(when we assume that the mapping is $L_{_{v}} \mapsto  v$).
Similarly, a graph $\matr{G}$ marked by a set
$\{(1,L_{_{1}}),(2,L_{_{2}}),\ldots,(n,L_{_{n}})\}$ is called an
{\it ordered list-graph} and is denoted by $\matr{G}[n,{\cal L}]$.\\
In all cases $\matr{G}$ will be referred to as the {\it base-graph}.
Also, when the list assignment ${\cal L}$ determines a proper
coloring $\gamma$ of $\matr{G}[n]$ (i.e. for any vertex $v$ we have
$|L_{_{v}}|=1$) then $\matr{G}[n,\gamma]$ is called an {\it ordered
colored-graph}\footnote{Note that considering the contents of the
previous section, and by abuse of notation, e.g.
$\matr{G}[\varrho,\gamma][u,v]$, stands for a {\it marked, ordered}
and {\it colored} graph where the ordering $\varrho$, the coloring
$\gamma$ and marks $u,v$ are clear from the context. }.
 }\end{defin}

\begin{defin}{\label{RULEBASE}
A {\it local $t$-coloring rule} ${\sf R}$ on the base-graph
$\matr{H}$ is an updating rule for the lists of a list-graph
$\matr{H}[{\cal L}]$ that changes it to a new list-graph
$\matr{H}[{\cal L}']$ in which the base-graph is not changed and,
\begin{itemize}
\item[{\rm a)}]{For each vertex $v \in V(\matr{H})$ we have $L'_{_{v}} \subseteq L_{_{v}}$.}
\item[{\rm b)}]{The list coloring problems $(\matr{H}[{\cal L}],t)$ and $(\matr{H}[{\cal L}'],t)$
have the same set of solutions.}
\item[{\rm c)}]{If using ${\sf R}$, $\matr{H}[{\cal L}_{_{1}}]$ is changed to $\matr{H}[{\cal L}'_{_{1}}]$,
$\matr{H}[{\cal L}_{_{2}}]$ is changed to $\matr{H}[{\cal
L}'_{_{2}}]$
 and ${\cal L}_{_{1}} \subseteq {\cal L}_{_{2}}$ then ${\cal L}'_{_{1}} \subseteq {\cal L}'_{_{2}}$.
}
\end{itemize}
Such a replacement of lists is usually denoted as
$${\sf R}: \quad  {\bf condition} \quad \Rightarrow \quad L \leftarrow L'$$
which means that the list $L$ is replaced by the list $L'$ if {\bf condition} is satisfied.
}\end{defin}
Note that a standard way of constructing such coloring rules is through constructing forcing structures, e.g. type~$1$ or type ~$2$
forcing introduced in Section~\ref{FORCING}. In what follows we formulate a couple of well-known rules based on what has already been discussed
in previous sections.
\begin{exm}{The following rule is called the Tucker's first rule, and is denoted by  ${\sf R}^{^{Tu}}_{_{1}}$ (see Section~\ref{PERFECT} for motivations).
\begin{equation}
   \left (\exists\ t > 1,\quad |A| \isdef |\bigcup_{x_{_{t}} \not = x_{_{i}} \in
   V(\matr{K}_{_{t}}[x_{_{1}},x_{_{2}},\ldots,x_{_{t}}]) } L_{_{x_{_{i}}}}|=t-1\right )
   \Rightarrow \left (L_{_{x_{_{t}}}} \leftarrow L_{_{x_{_{t}}}} - A \right ).
\end{equation}
Note that Tucker's first rule is a type~$1$ forcing on a clique.
Also, define $\matr{T}_{_{t}}^{^2}[u,v]$ as follows,
$$\matr{T}_{_{t}}^{^2}[u,v,x_{_{1}},x_{_{2}},\ldots,x_{_{t-1}}] \isdef \matr{K}_{_{t}}[u,x_{_{1}},x_{_{2}},\ldots,x_{_{t-1}}]+
\matr{K}_{_{t}}[v,x_{_{1}},x_{_{2}},\ldots,x_{_{t-1}}].$$
 Also, the following rule is called the Tucker's second rule, and is denoted by  ${\sf
 R}^{^{Tu}}_{_{2}}$ (see Section~\ref{PERFECT} for motivations).
\begin{equation}
   \left (\exists\ t > 1,\quad |\bigcup_{w \in V(\matr{T}_{_{t}}^{^2}[u,v,x_{_{1}},x_{_{2}},\ldots,x_{_{t-1}}])} L_{_{w}}|=t\right ) \Rightarrow
   \left (L_{_{u}} \leftarrow L_{_{u}} \cap L_{_{v}}\right ).
\end{equation}
It is also clear that Tucker's second rule is a type~$2$ forcing on a clique.
}\end{exm}
Hereafter, we adopt the notation ${\cal R}_{_{T}}  \isdef \{{\sf R}^{^{Tu}}_{_{1}},{\sf R}^{^{Tu}}_{_{2}}\}$.
Moreover, we refer to a such an specific rule as ${\sf R}^{^{Tu}}_{_{1}}(t_{_{0}})$ for some $t=t_{_{0}}$.
It is also instructive to consider the special case of $t=2$ for both Tucker rules. Note that in the first case, ${\sf R}^{^{Tu}}_{_{1}}(2)$
states (the well-known and simple rule) that, in any coloring problem, if one end of an edge has a fixed color, then this color may be excluded from the list of the other end of this edge. This can be considered as the most natural coloring rule coming directly from the definition of a {\it proper coloring} of a graph, and consequently, from now on, we assume that  ${\sf R}^{^{Tu}}_{_{1}}(2)$ is contained in any coloring rule-base.\\
Also, ${\sf  R}^{^{Tu}}_{_{2}}(2)$ states that in a $2$-coloring of a path of length two, both vertices of degree one must
have the same color.
\begin{defin}{
A {\it $t$-coloring rule-base} ${\cal R}$, with $r \geq 0$ {\it
structural rules} and the  {\it nonstructural rule} ${\sf R}^*$, is
a finite set
$${\cal R} \isdef \{{\sf
R}_{_{1}},{\sf R}_{_{2}},\ldots,{\sf R}_{_{r}}\} \cup
\{{\sf R}^*\}$$ such that,
\begin{itemize}
\item{For each $1 \leq i \leq r$, ${\sf R}_{_{i}}$ is a local $t$-coloring rule on a base-graph $\matr{H}_{_{i}}$.}
\item{The nonstructural rule ${\sf R}^*$ is a coloring procedure to change a list $L \subseteq \{1,2,\ldots,t\}$ to a list
$L' \subseteq L$ that must be compatible with ${\sf
R}^{^{Tu}}_{_{1}}(2)$. This rule may not satisfy properties $(b)$ and $(c)$ of Definition~\ref{RULEBASE}. }
\end{itemize}
A $t$-coloring rule-base is said to be {\it $d$-bounded}, if for
every $1 \leq i \leq r$, the diameter of the base-graph
$\matr{H}_{_{i}}$, of the local rule ${\sf R}_{_{i}}$, is not
greater than the constant $d$.
A $t$-coloring rule-base is said to be {\it structural}, if it
consists, only,  of structural rules.
 }\end{defin}
\begin{exm}{A standard example for the nonstructural coloring rule ${\sf R}^*$,
is the {\it greedy} coloring rule, ${\sf R}^{^{Gr}}$, which is
defined as follows,
\begin{equation}
{\sf R}^{^{Gr}}:     \left (A \isdef \displaystyle{\bigcup_{{x \in N(u) \ \& \ |L_{_{x}}|=1}}}\ \ L_{_{x}}  \right ) \Rightarrow \left (L_{_{u}} \leftarrow \{\min(L_{_{u}}-A)\} \right ).
\end{equation}
Hereafter, ${\cal R}_{_{G}}  \isdef \{{\sf R}^{^{Gr}}\}$.
}\end{exm}

\begin{defin}{
If ${\cal R}$ is a $d$-bounded $t$-coloring rule-base, then the list
coloring problem $(\matr{G}[n],{\cal L},t)$ is said to be {\it
$({\cal R},d,r)$-solvable} if the following algorithm,
\begin{center}
 {\bf OLG-${\cal R}$-Col}: $(\matr{G}[n],{\cal L},t,d,rounds,Done)$,
\end{center}
 ends
up with a $t$-coloring of $\matr{G}[n]$ as a solution with
$Done=true$, $rounds < r$, where $rounds$ returns the number of
rounds that the algorithm has scanned the graph and the {\it
returned} list assignment ${\cal L}$ is a proper coloring of
$\matr{G}[n]$ such that for any
vertex $v$ we have $|L_{_{v}}|=1$.\\
\ \\
\begin{tabular}{l}
{\bf OLG-${\cal R}$-Col}: $(\matr{G}[n],{\cal L},t,d,rounds,Done)$\\
\ \\
\begin{tabular}{ll}
\hfill $0)$& $rounds:=0, Done:=True $, \\
\hfill $1)$&$colored:=True$,\\
\hfill $2)$&{\bf For} $i$ from $1$ to $n$ do\\
&\quad  {\bf Local-${\cal R}$-Update}$(\matr{G}[n],{\cal L},t,v_{_{i}},d,Done,Col)$,\\
&\quad {\bf If } (NOT $Done$) {\bf Then} Halt and write(Failure), \\
&\quad {\bf If } (NOT $Col$) {\bf Then} $colored:=False$, \\
&{\bf End} (For),\\
\hfill $3)$& $rounds:=rounds+1$,\\
\hfill $4)$& {\bf If}  (NOT $colored$) {\bf Then} goto: $(1)$\\
&{\bf Else} Halt and write the coloring,
\end{tabular}\\ \ \\
{\bf End}\ (OLG-${\cal R}$-Col).\\
\end{tabular}\\ \ \\
The procedure {\bf Local-${\cal R}$-Update} locally updates the
lists according to the rule-base of the algorithm and, moreover, we
assume that this is done in a {\it canonical order}. In more
details we have,\\ \ \\
\begin{tabular}{l}
{\bf Local-${\cal R}$-Update}: $(\matr{G}[n],{\cal L},t,v,d,Done,Col)$\\
\ \\
\begin{tabular}{ll}

\hfill $0)$&$Col:=True$,\\
\hfill $1)$&Focus on a $d$-neighborhood of $v$ in
$\matr{G}$ and update the list of $v$ using\\
& the local rules (${\sf R}_{_{i}}$'s) in ${\cal R}$ and all subgraphs of the neighborhood\\
& (in a canonical ordering).\\
\hfill $2)$& Update the list of $v$ using ${\sf R}^*$ (if any such rule exists in the rule-base),\\
\hfill $3)$&{\bf If} $|L(v)|=0$ {\bf Then} $Done:=False$,\\
\hfill $4)$&{\bf If} $|L(v)| \not = 1$ {\bf Then} $Col:=False$,\\
\end{tabular}\\ \ \\
{\bf End}\ (Local-${\cal R}$-Update).
\end{tabular}
 \\ \ \\
 A list coloring problem $(\matr{G}[n],{\cal L},t)$ is said to be
{\it $({\cal R},d)$-solvable} if it is $({\cal
R},d,\infty)$-solvable (which means that it is $({\cal
R},d,r)$-solvable for some natural number $r$).
 }\end{defin}
\subsection{Sequential defining sets}
\begin{defin}{
Given a $t$-coloring $\gamma$ of a graph  $\matr{G}$, and a subset
$A \subset V(G)$ then the list assignment ${\cal L}^{^{A,\gamma}}$ is
defined as follows,
$${\cal L}^{^{A,\gamma}}_{_{v}} \isdef \left \{ \begin{array}{ll}
\{\gamma(v)\} & v \in A \\
\{1,2,\ldots,t\} & v \not \in A. \end{array}\right. $$ Given an
ordered colored-graph $\matr{G}[n,\gamma]$, a subset $A \in V(G)$ is
called an {\it sequential ${\cal R}$-defining set} (or an
$\sds_{_{{\cal R}}}$ for short) for $\matr{G}[n,\gamma]$ with
respect to the  $d$-bounded $t$-coloring rule-base ${\cal R}$, if
the list coloring problem $(\matr{G}[n],{\cal L}^{^ {A,\gamma}},t)$
is $({\cal R},d)$-solvable, and returns the coloring $\gamma$.
 In that case, the number
$$\iota_{_{\matr{G}[n,\gamma],{\cal R}}}(A) \isdef |A|+rounds-1 $$ is called the {\it index} of $A$ with respect to $\matr{G}[n,\gamma]$.\\
The subset $A \subseteq V(G)$ is called a {\it weak minimum sequential ${\cal
R}$-defining set} (or a $\Wsds_{_{{\cal R}}}$) for
$\matr{G}[n,\gamma]$ with respect to the  $d$-bounded $t$-coloring
rule-base ${\cal R}$ (or a $\Wsds(\matr{G}[n],\gamma,{\cal R})$ for short), if $A$ is an $\sds_{_{{\cal R}}}$   for $\matr{G}[n,\gamma]$
of minimum size.\\
Also, a {\it strong minimum sequential ${\cal R}$-defining set} (or a $\Ssds_{_{{\cal R}}}$) for $\matr{G}[n,\gamma]$ with
respect to the  $d$-bounded $t$-coloring rule-base ${\cal R}$ (or a $\Ssds(\matr{G}[n],\gamma,{\cal R})$ for short),
is an $\sds_{_{{\cal R}}}$   for $\matr{G}[n,\gamma]$ of minimum  index. The minimum numbers are called
the {\it weak sequential defining number} and the {\it strong sequential  defining number}
of $\matr{G}[n,\gamma]$ with respect to the  $d$-bounded
$t$-coloring rule-base ${\cal R}$, and are denoted by
$\Wsdn(\matr{G}[n],\gamma,{\cal R})$ and $\Ssdn(\matr{G}[n],\gamma,{\cal
R})$, respectively.\\
On the other hand, the bounded versions of these parameters are
defined similarly as the minimum numbers of the size or index,
respectively, on the space of all {\it $k$-bounded   sequential
${\cal R}$-defining   sets}, i.e. the subsets $A$ for which the
sequential coloring algorithm {\bf OLG-${\cal
R}$-Col}$(\matr{G}[n],{\cal L}^{^ A},t,d,rounds,Done)$ succeeds and
ends with $Done=true$ and $rounds \leq k$. These numbers are denoted
as $\Wsdn^{k}(\matr{G}[n],\gamma,{\cal R})$ and
$\Ssdn^{k}(\matr{G}[n],\gamma,{\cal R})$, respectively. Also,
$\sds^{k}_{_{{\cal R}}}$, $\Wsds^{k}_{_{{\cal R}}}$ and
$\Ssds^{k}_{_{{\cal R}}}$ are defined accordingly.
 }\end{defin} In
what follows we consider some basic properties of these concepts.
\begin{pro}\label{ALGDET}
Given any ordered $t$-colored graph $\matr{G}[\varrho,\gamma]$, any
$d$-bounded $t$-coloring rule-base ${\cal R}$, and $A \subset V(G)$,
if the list coloring problem $(\matr{G}[\varrho],{\cal
L}^{^{A,\gamma}},t)$ is $({\cal R},d)$-solvable then it is also an
$({\cal R},d,\|{\cal L}^{^{A,\gamma}}\|-|V(\matr{G})|)$-solvable
problem.
\end{pro}
\begin{proof}{Clearly, if the algorithm does not stop after $\|{\cal L}^{^{A,\gamma}}\|-|V(\matr{G})|$ rounds, it means that there has been a round
during which no list is changed and, since {\bf OLG-${\cal R}$-Col}
is a deterministic algorithm, this means that after this round no list can be  changed.
}\end{proof}
\begin{pro}
Given any $t$-colored graph $\matr{G}[\gamma]$,
any $d$-bounded $t$-coloring structural rule-base ${\cal R}$, with
two orderings $\varrho_{_{1}}$ and $\varrho_{_{2}}$ for $\matr{G}$,
if $A$ is a $\Wsds_{_{{\cal R}}}$ for $\matr{G}[\varrho_{_{1}},\gamma]$, then $A$ is also
a $\Wsds_{_{{\cal R}}}$ for $\matr{G}[\varrho_{_{2}},\gamma]$.
\end{pro}
\begin{proof}{
Let $S \isdef ((t_{_{i}},v_{_{i}},\matr{X}_{_{i}}))_{_{i \in \{1,2,\ldots,s\}}}$
be the sequence of vertices, as $v_{_{i}}$, along with their $d$-neighborhoods
$\matr{X}_{_{i}}$ (as an ordered list-graph), such that
$v_{_{i}}$ is the $i$'th vertex whose list is changed (made smaller) by the algorithm
\begin{center}
{\bf OLG-${\cal R}$-Col}$(\matr{G}[\varrho_{_{1}}],{\cal L}^{^
{A,\gamma}},t,d,rounds,Done).$
\end{center}
at time $t_{_{i}}$.
Note that this is a well-defined sequence since by the hypothesis at least one list is changed in each round
(or the algorithm will fall in a loop-state which is a contradiction).\\
Consider the same sequence $S' \isdef ((t'_{_{i}},v'_{_{i}},\matr{X}'_{_{i}}))_{_{i \in \{1,2,\ldots,s'\}}}$ for the algorithm
\begin{center}
{\bf OLG-${\cal R}$-Col}$(\matr{G}[\varrho_{_{2}}],{\cal L}^{^
{A,\gamma}},t,d,rounds,Done).$
\end{center}
Let $f$ be the increasing function that presents the best-fit embedding of $S$ in $S'$, i.e.
 $f(i)$ is the smallest index in the sequence $S'$ such that $v'_{_{f(i)}}=v_{_{i}}$.
 Now since the rules in ${\cal R}$ does not change the set of possible colorings of the neighborhood structure they are applied to
 (and consequently does not change the set of possible colorings of the whole graph), each list of the graph $\matr{X}'_{_{f(i)}}$
 is a subsets of the corresponding list in the graph $\matr{X}_{_{i}}$ for each $i$. This proves that the second algorithm
 for $\matr{G}[\varrho_{_{2}}]$ will eventually reach the same coloring $\gamma$ but possibly in a longer period of time bounded by
 the function $t'_{_{f(i)}}$.
 }\end{proof}
  It is important to note that the condition of being
structural for the rule-base in Proposition~\ref{ALGDET} can not be
removed since, for instance, on
$\matr{K}_{_{3}}[v_{_{1}},v_{_{2}},v_{_{3}}]$ with $v_{_{1}}$
colored red one can easily see that by reversing the order on the
other two vertices the whole coloring is changed when using the
greedy sequential algorithm.
\begin{pro}\label{basics}
Given any ordered $t$-colored graph $\matr{G}[n,\gamma]$, and any
$d$-bounded
 \linebreak
$t$-coloring rule-base ${\cal R}$, with
$\Ssdn(\matr{G}[n],\gamma,{\cal R}) < \infty$, then
\begin{itemize}
\item[{\rm a)}]{ $\forall\ k \geq 1 \quad \Wsdn^{k}(\matr{G}[n],\gamma,{\cal R})
\leq \Ssdn^{k}(\matr{G}[n],\gamma,{\cal R})$.}

\item[{\rm b)}]{A subset $A$ is a $\Ssds^{1}_{_{{\cal R}}}$ for $\matr{G}[n,\gamma]$
if and only if it is a $\Wsds^{1}_{_{{\cal R}}}$, and consequently,
$\Ssdn^{1}(\matr{G}[n],\gamma,{\cal R})=\Wsdn^{1}(\matr{G}[n],\gamma,{\cal R})$. }

\item[{\rm c)}]{If  $1 \leq i \leq j$ then
$$ \Ssdn(\matr{G}[n],\gamma,{\cal R}) \leq \Ssdn^{j}(\matr{G}[n],\gamma,{\cal R})
\leq \Ssdn^{i}(\matr{G}[n],\gamma,{\cal R}) \leq n.$$
Moreover, the upper bound is tight.}

\item[{\rm d)}]{If  $1 \leq i \leq j$ then
$$\Wsdn(\matr{G}[n],\gamma,{\cal R}) \leq \Wsdn^{j}(\matr{G}[n],\gamma,{\cal R})
\leq \Wsdn^{i}(\matr{G}[n],\gamma,{\cal R}).$$
Moreover, if ${\cal R}$ is a structural rule-base then,
$$t-1 \leq \Wsdn(\matr{G}[n],\gamma,{\cal R})$$
and the lower bound is tight.}
\end{itemize}
\end{pro}
\begin{proof}{\ \\
\begin{itemize}
\item[{\rm a)}]{ If $A$ is a $\Wsds^{k}_{_{{\cal R}}}$ for $\matr{G}[n,\gamma]$, then
the list coloring problem $(\matr{G}[n],{\cal L}^{^ {A,\gamma}},t)$
is $({\cal R},d,k+1)$-solvable.
 Any other set $B$ that is an
$\Ssds^{k}_{_{{\cal R}}}$ for $\matr{G}[n,\gamma]$ must have a
greater cardinality than $A$ due to the definition of
$\Wsds^{k}_{_{{\cal R}}}$, since it takes at least $1$ round for
\begin{center}
{\bf
OLG-${\cal R}$-Col}$(\matr{G}[n],{\cal L}^{^
{B,\gamma}},t,d,rounds,Done)$
\end{center}
to be finished and $\iota_{_{\matr{G}[n,\gamma],{\cal R}}}(B)=|B|+rounds-1$.
}
\item[{\rm b)}]{
If $A$ is a $\Ssds^{1}_{_{{\cal R}}}$ for $\matr{G}[n,\gamma]$ then
the list coloring problem $(\matr{G}[n],{\cal L}^{^ {A,\gamma}},t)$
is $({\cal R},d)$-solvable in exactly one round. Therefore,
$$\iota_{_{\matr{G}[n,\gamma],{\cal R}}}(A)=|A|+rounds-1=|A|$$ and
from the definition, one deduces that $A$ is also a
$\Wsds^{1}_{_{{\cal R}}}$. The converse implication follows in the
same way. }
\item[{\rm c)}]{If $A$ is an $\Ssds^{k}_{_{{\cal R}}}$ for $\matr{G}[n,\gamma]$ then
the list coloring problem $(\matr{G}[n],{\cal L}^{^ {A,\gamma}},t)$
is $({\cal R},d,k+1)$-solvable. Also, note that for the empty graph
of size $n$ we have
$$\Ssdn^{1}(\overline{\matr{K}}_{_{n}}[n],\gamma,{\cal R}_{_{T}}) = n,$$
with respect to any $t$-coloring $\gamma$.
 }
\item[{\rm d)}]{Let $A$ be a $\Wsds_{_{{\cal R}}}$
for $\matr{G}[n,\gamma]$ and the structural rule-base ${\cal R}$. Then, by definition $\matr{G}[n,{\cal L}^{^
{A,\gamma}}]$ is a uniquely list $t$-colorable graph and
consequently, $|A| \geq t-1$. \\
Also, $\Wsdn(\matr{K}_{_{3}}[n],\gamma,{\cal R}_{_{T}})=2$ for any $3$-coloring $\gamma$ which shows that the inequality is sharp.}
\end{itemize}
}\end{proof}

\begin{defin} { \label{DnGraph}
For each $k \geq 1$ the ordered marked graph
$\matr{D}_{_{k}}[\varrho][u,v]$ is defined as follows (see
Figure~\ref{GraphD}). For $k=1$ we define,
\begin{equation}
\begin{array}{ll}
\matr{D}_{_{1}}[\varrho][u,v,z]& \isdef
\matr{\varepsilon}[u,z]+\matr{\varepsilon}[v,z],
\end{array}
\end{equation}
with
$$ \varrho \isdef \left \{ \begin{array}{ll}
\varrho(u)=1,\\
\varrho(v)=2,  \\
\varrho(z)=3,
\end{array} \right. $$
and for $k \geq 2$,
\begin{equation}
\begin{array}{ll}
\matr{D}_{_{k}}[\varrho][u,v]&\isdef \matr{\varepsilon}[u,z]+ \matr{\varepsilon}[u,x_{_{k}}]+\matr{\varepsilon}[v,x_{_{k}}]+
\matr{\varepsilon}[z,x_{_{1}}]\\
&\\
&\ + \ \displaystyle{\sum^{{k-1}}_{{i=1}}}\ \matr{K}_{_{3}}[x_{_i},v_{_{1,i}},v_{_{2,i}}]+\matr{K}_{_{3}}[x_{_{i+1}},v_{_{1,i}},v_{_{2,i}}] \\
&\\
&\ + \ \displaystyle{\sum^{k-1}_{i=1}}\
{\matr{\varepsilon}[z,v_{_{1,i}}]},
\end{array}
\end{equation}
with the ordering
$$ \varrho \isdef \left \{ \begin{array}{lllll}
\varrho(u)=1, \\
\varrho(v)=2,  \\
\varrho(z)=k+3, \\
\varrho(x_{_i})=i+2 & \forall\ i, \\
\varrho(v_{_{1,i}})=k+2i+2 & \forall\ i, \\
\varrho(v_{_{2,i}})=k+2i+3 & \forall\ i.
\end{array} \right. $$
}
\end{defin}

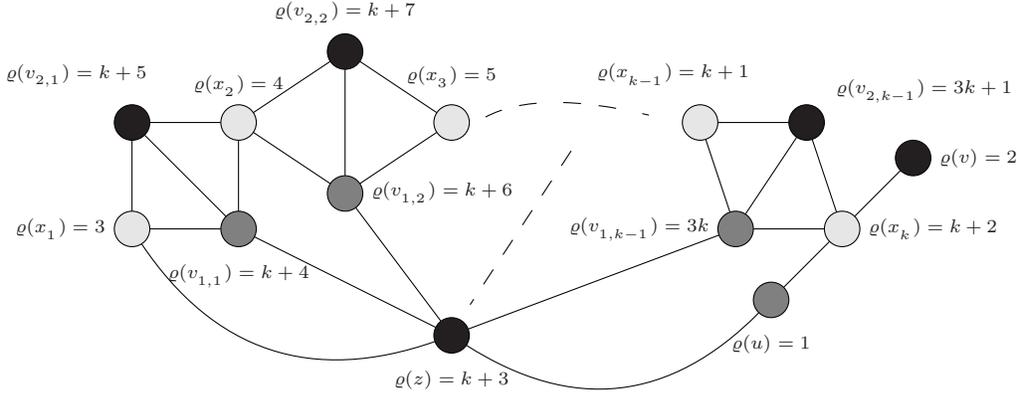
\begin{figure}[h]

\fontsize{7.5 pt}{7.5 pt}
\special{em:linewidth 1 pt} \unitlength 1.18mm \linethickness{0.5pt}
\begin{picture}(110,48)(0,-48)
\put(0,-122){}
\gasset{Nw=4.0,Nh=4.0,Nmr=2.0,Nfill=y,ExtNL=y,NLdist=1.0,AHnb=0}

\node[Nfill=y,fillcolor=Black,ExtNL=y,NLangle=0.0,NLdist=1.0](n4)(104.0,-16.0){$\varrho(v)=2$}

\node[Nfill=y,fillgray=0.5,ExtNL=y,NLangle=-90.0,NLdist=2.0](n5)(88.0,-32.0){$\varrho(u)=1$}

\node[Nfill=y,fillgray=0.9,ExtNL=y,NLangle=0.0,NLdist=1.0](n6)(96.0,-24.0){$\varrho(x_{_{k}})=k+2$}

\node[Nfill=y,fillgray=0.9,ExtNL=y,NLangle=120.0,NLdist=2.0](n7)(80.0,-12.0){$\varrho(x_{_{k-1}})=k+1$}

\node[Nfill=y,fillcolor=Black,ExtNL=y,NLangle=15.0,NLdist=2](n8)(92.0,-12.0){$\varrho(v_{_{2,k-1}})=3k+1$}

\node[Nfill=y,fillgray=0.5,ExtNL=y,NLangle=180.0,NLdist=1.0](n9)(84.0,-24.0){$\varrho(v_{_{1,k-1}})=3k$}

\node[Nfill=y,fillcolor=Black,ExtNL=y,NLangle=140.0,NLdist=2.0](n10)(16.0,-12.0){$\varrho(v_{_{2,1}})=k+5$}

\node[Nfill=y,fillgray=0.9,ExtNL=y,NLdist=1.0](n11)(28.0,-12.0){$\varrho(x_{_{2}})=4$}

\node[Nfill=y,fillgray=0.5,ExtNL=y,NLangle=-90.0,NLdist=2.0](n12)(28.0,-24.0){$\varrho(v_{_{1,1}})=k+4$}

\node[Nfill=y,fillgray=0.9,ExtNL=y,NLangle=180.0,NLdist=1.0](n13)(16.0,-24.0){$\varrho(x_{_{1}})=3$}

\node[Nfill=y,fillcolor=Black,ExtNL=y,NLangle=-90.0,NLdist=2.0](n14)(52.0,-36.0){$\varrho(z)=k+3$}

\node[Nfill=y,fillcolor=Black,ExtNL=y,NLangle=90.0,NLdist=1.0](n15)(40.0,-4.0){$\varrho(v_{_{2,2}})=k+7$}

\node[Nfill=y,fillgray=0.5,ExtNL=y,NLangle=0.0,NLdist=1.0](n16)(40.0,-20.0){$\varrho(v_{_{1,2}})=k+6$}

\node[Nfill=y,fillgray=0.9,ExtNL=y,NLdist=2.0](n17)(52.0,-12.0){$\varrho(x_{_{3}})=5$}

\gasset{AHnb=0} \drawedge(n10,n12){}

\drawedge(n5,n6){}

\drawedge(n4,n6){ }

\drawedge(n8,n6){}

\drawedge(n9,n6){}

\drawedge(n7,n8){}

\drawedge(n7,n9){}

\drawedge(n11,n12){}

\drawedge(n13,n12){}

\drawedge(n13,n10){}

\drawedge(n11,n10){}

\drawedge(n15,n16){}

\drawedge(n11,n16){}

\drawedge(n15,n11){}

\drawedge(n17,n15){}

\drawedge(n17,n16){}

\drawedge[curvedepth=-8.0](n13,n14){}

\drawedge[curvedepth=-8.0](n14,n5){}

\drawedge(n12,n14){}

\drawedge(n16,n14){}

\drawedge(n9,n14){}

\drawedge(n8,n9){}
\drawqbezier[dash={2.0 2.0 2.0
3.0}{0.0}](55.83,-11.38,61.82,-8.2,74.23,-11.17)
\drawline[dash={2.0 2.0 2.0 2.0}{0.0}](54.09,-32.54)(65.47,-15.2)
\end{picture}
\caption{\protect\label{GraphD} The graph
$\matr{D}_{_{k}}[\varrho][u,v]$.  }
\end{figure}
\begin{defin}{
\label{CCK} The $k$'th {\it coloring closure} of a subset
$A$ of $\matr{G}[\varrho,\gamma]$, with respect to a $d$-bounded
$t$-coloring rule-base ${\cal R}$, denoted by
$CC^{k}(A,\matr{G}[\varrho,\gamma],{\cal R})$, is defined to be the
induced subgraph $\matr{H}[\tilde{\varrho},\tilde{\gamma}]$ of
$\matr{G}[\varrho,\gamma]$ with the induced ordering
$\tilde{\varrho}$, and the induced coloring $\tilde{\gamma}$, on
all vertices that has a list of size one, when the algorithm
\begin{center}
{\bf OLG-${\cal R}$-Col}$(\matr{G}[\varrho,\gamma],{\cal L}^{^
{A,\gamma}},t,d,rounds,Done)$
\end{center}
is forced to stop after $k$ rounds (if it has not already halted).
 }\end{defin}
The following corollary is a direct consequence of the definition which is stated to be used and referred to
in later applications.
\begin{cor}\label{CCK_COL}
Given any $d$-bounded $t$-coloring rule-base ${\cal R}$, and a
graph $\matr{G}[\varrho,\gamma]$,
\begin{itemize}
\item[{\rm a)}]{$CC^{\infty}(A,\matr{G}[\varrho,\gamma],{\cal R})$ is well-defined.}
\item[{\rm b)}]{Let $k \geq 1$ be an integer and let $V(\matr{G})$ be $2$-partitioned as $V(\matr{G})=V(\matr{G}_{_{1}})\cup V(\matr{G}_{_{2}})$
with $V(\matr{G}_{_{1}}) \cap V(\matr{G}_{_{2}})=\emptyset$, where
$\matr{G}_{_{1}}[\varrho_{_{1}},\gamma_{_{1}}]$ and
$\matr{G}_{_{2}}[\varrho_{_{2}},\gamma_{_{2}}]$ are induced
subgraphs on $V(\matr{G}_{_{1}})$ and $V(\matr{G}_{_{2}})$ with the
induced orderings and colorings, respectively. Moreover, assume
that,

   \begin{itemize}
     \item[{\rm $\star$)}]{If $A$ is either a $\Wsds^{k}(\matr{G_{_{1}}}[\varrho_{_{1}}],\gamma_{_{1}},{\cal R})$,
     or a $\Wsds^{k}(\matr{G_{_{2}}}[\varrho_{_{2}}],\gamma_{_{2}},{\cal
     R})$,
     then as induced subgraphs
                  $$CC^{k}(A,\matr{G}[\varrho,\gamma],{\cal R})\neq \matr{G},$$}
   \end{itemize}
then,
$$\Wsdn^{k}(\matr{G }[n],\gamma,{\cal R}) > \max \{\Wsdn^{k}(\matr{G_{_{1}} }[n],\gamma,{\cal R}),
\Wsdn^{k}(\matr{G_{_{2}} }[n],\gamma,{\cal R}) \}.$$
 }
\end{itemize}
\end{cor}
\begin{proof}{For part $(a)$ note that by Proposition~\ref{ALGDET} if $r_{_{0}} \isdef \|{\cal L}^{^{A,\gamma}}\|-|V(\matr{G})|$ we have,
$$ CC^{\infty}(A,\matr{G}[\varrho,\gamma],{\cal R})=CC^{r_{_{0}}}(A,\matr{G}[\varrho,\gamma],{\cal R}). $$
Part $(b)$ follows from the definitions and the hypothesis  $(\star)$.
 }\end{proof}
\begin{thm}\label{WSDIF}
For any given integer $k \geq 2$ there exist an ordered colored graph
$\matr{G}[\varrho,\gamma]$ and a rule-base ${\cal R}$ such that,
$$\Ssdn^{k}(\matr{G}[\varrho],\gamma,{\cal R}) \not = \Wsdn^{k}(\matr{G}[\varrho],\gamma,{\cal R}).$$
\end{thm}
\begin{proof}{
Fix $k \geq 2$ and consider the graph
$\matr{D}_{_{k}}[\varrho][u,v]$ as defined in
Definition~\ref{DnGraph} (see Figure~\ref{GraphD}). Also, consider
the $3$-coloring $\gamma$ of $\matr{D}_{_{k}}$ defined as,
$$ \gamma \isdef \left \{ \begin{array}{lllll}
\gamma(u)=3,\\
\gamma(v)=1,  \\
\gamma(z)=1, \\
\gamma(x_{_i})=2 & \forall\ i, \\
\gamma(v_{1,i})=3 & \forall\ i,  \\
\gamma(v_{2,i})=1 & \forall\ i .
\end{array} \right. $$
Let $A_{_{1}} \isdef \{u,v\}$. Then it is easy to see that the
list coloring problem $(\matr{D}_{_{k}}[\varrho],{\cal
L}^{^{A_{_{1}},\gamma}},3)$ is $({\cal R}_{_{T}},2)$-solvable
in exactly $k$ rounds with the coloring $\gamma$ as its solution.
Also, it is easy to see that the vertex $v$ can not be colored by
any one of the rules in ${\cal R}_{_{T}}$, and consequently, must be
contained in any $\sds_{_{{\cal R}_{_{T}}}}$ of $\gamma$. Hence,
$$CC^{^{\infty}}(\{v\},\matr{D}_{_{k}}[\varrho,\gamma],{\cal R}_{_{T}}) \not = \matr{D}_{_{k}}[\varrho,\gamma] \quad  \Rightarrow
\quad  \Wsdn^{k}(\matr{D}_{_{k}}[\varrho],\gamma,{\cal R}_{_{T}})=2.$$
On the other hand, note that for $A_{_{2}} \isdef \{u,v,x_{_{1}}\}$,
the list coloring problem
\linebreak
$(\matr{D}_{_{k}}[\varrho],{\cal
L}^{^{A_{_{2}},\gamma}},3)$ is $({\cal R}_{_{T}},2)$-solvable
in exactly $1$ round with the coloring $\gamma$ as its solution,
which shows that
$$\Ssdn^{k}(\matr{D}_{_{k}}[\varrho],\gamma,{\cal R}_{_{T}}) \leq 3.$$
Again, by considering the vertex $v$, it is easy to see that
$$\Ssdn^{k}(\matr{D}_{_{k}}[\varrho],\gamma,{\cal R}_{_{T}}) \geq 2,$$
and that $v$ must be in any such $\Ssds_{_{{\cal R}_{_{T}}}}$.
Moreover, it is easy to check that
$$CC^{k}(\{v,x\},\matr{D}_{_{k}}[\varrho,\gamma],{\cal R}_{_{T}}) \not = \matr{D}_{_{k}}[\varrho,\gamma],$$
for any vertex $x \in V(\matr{D}_{_{k}})$, which shows that
$$\Ssdn^{k}(\matr{D}_{_{k}}[\varrho],\gamma,{\cal R}_{_{T}})=3.$$
}\end{proof}

\begin{thm}\label{NONTRIVSPEC}
Given integers
 $\xi \geq 1$ and $k \geq 1$, there exist ordered colored graphs $\matr{G}[\varrho,\gamma]$ and
 $\matr{H}[\varrho',\gamma']$, and also  rule-bases ${\cal R}$ and  ${\cal R'}$  such that
$$\Wsdn^{k}(\matr{G}[\varrho],\gamma,{\cal R})-\Wsdn^{k+1}(\matr{G}[\varrho],\gamma,{\cal R})=\xi,$$
and
$$\Ssdn^{k}(\matr{H}[\varrho'],\gamma',{\cal R'})-\Ssdn^{k+1}(\matr{H}[\varrho'],\gamma',{\cal R'})=\xi.$$
\end{thm}
\begin{proof}{
For each $\xi \geq 1$ consider the following ordered marked
graph,
$$\matr{G}_{_{\xi}}[\rho][u,v] \isdef \sum_{j=1}^{\xi}{\matr{D}^{j}_{_{k+1}}[\varrho][u,v]},$$
where each the ordering $\varrho$ is defined in
Definition~\ref{DnGraph}, and the ordering $\rho$ is given by,
$$ \rho(w) \isdef \left \{ \begin{array}{lll}
1 & w=u,\\
2 & w=v,\\
(3k-1)(j-1)+\varrho(w) &  \{u,v\} \not \ni w \in V(\matr{D}^{j}_{_{k+1}}[u,v,z]).
\end{array} \right. $$
Also, consider the following $3$-coloring $\gamma_{_{1}}$ for each
${\matr{D}^{j}_{_{k+1}}[\varrho][u,v]}$,
$$ \gamma_{_{1}} \isdef \left \{ \begin{array}{lll}
\gamma_{_{1}}(u)=3,\\
\gamma_{_{1}}(v)=1,  \\
\gamma_{_{1}}(z)=1, \\
\gamma_{_{1}}(x_{_i})=2 & \forall \ i,\\
\gamma_{_{1}}(v_{_{1,i}})=3 & \forall \ i,\\
\gamma_{_{1}}(v_{_{2,i}})=1 & \forall \ i.\\
\end{array} \right. $$
Define the coloring $\gamma$ on $\matr{G}_{_{\xi}}[\rho]$ to be the
coloring whose restriction on each copy of
${\matr{D}^{j}_{_{k+1}}[\varrho][u,v]}$ is $\gamma_{_{1}}$. By
considering the vertex $v$ and the fact that it must be contained in
any $\sds_{_{{\cal R}_{_{T}}}}$ for $\gamma$, one may check that
$$\Wsdn^{k+1}(\matr{G}_{_{\xi}}[\rho],\gamma,{\cal R}_{_{T}}) \geq 2.$$
Also, by setting $A=\{u,v\}$, it is easy to see that the list
coloring problem \linebreak $(\matr{G}_{_{\xi}}[\rho],{\cal L}^{^
{A,\gamma}},3)$ is $({\cal R}_{_{T}},2)$-solvable in exactly $k+1$
round and returns the coloring $\gamma$ as its solution,
 which proves that $A$ is a
$\Wsds^{k+1}(\matr{G}_{_{\xi}}[\rho],\gamma,{\cal R}_{_{T}})$, and
consequently,
$$\Wsdn^{k+1}(\matr{G}_{_{\xi}}[\rho],\gamma,{\cal R}_{_{T}})=2.$$
On the other hand, we prove
$\Wsdn^{k}(\matr{G}_{_{\xi}}[\rho],\gamma,{\cal R}_{_{T}})=\xi+2,$
using induction on $\xi$.\\
For the initial case $\xi=1$ we must prove that
$$ \Wsdn^{k}(\matr{D}_{_{k+1}}[\varrho],\gamma_{_{1}},{\cal R}_{_{T}})=3, $$
which follows from the facts that
$$\forall \ z \in V(\matr{D}_{_{k+1}})\quad CC^{k}(\{v,z\},\matr{D}_{_{k+1}}[\varrho,\gamma_{_{1}}],{\cal R}_{_{T}})
\not = \matr{D}_{_{k+1}}[\varrho,\gamma_{_{1}}], $$ and
$$ CC^{k}(\{v,u,x_{_{1}}\},\matr{D}_{_{k+1}}[\varrho,\gamma_{_{1}}],{\cal R}_{_{T}}) = \matr{D}_{_{k+1}}[\varrho,\gamma_{_{1}}].$$
 Therefore, we assume that the
claim is true for $\xi=t$, i.e.
$$\Wsdn^{k}(\matr{G}_{_{t}}[\rho],\gamma,{\cal R}_{_{T}})=t+2.$$ Let
$\matr{X}_{_{1}}$ be the induced subgraph of
$\matr{G}_{_{t+1}}[\rho,\gamma]$ on the vertex set of the copy
$\matr{G}_{_{t}}[\rho,\gamma]$, and also let $\matr{X}_{_{2}}$ be
the the induce subgraph, formed on the vertices \linebreak
$V(\matr{G}_{_{t+1}}[\rho,\gamma]) - V(\matr{X}_{_{1}})$. Note that
these two subgraphs satisfy the conditions of
Corollary~\ref{CCK_COL} and let $B_{_{1}}$ be a
$\Wsds^{k}(\matr{X}_{_{1}},\gamma|_{_{{\matr{X}_{_{1}}}}},{\cal
R}_{_{T}}).$ Then it is easy to check that
$$CC^{k}(B_{_{1}},\matr{G}_{_{t+1}}[\rho,\gamma],{\cal R}_{_{T}})
\subsetneq \matr{G}_{_{t+1}}[\rho,\gamma].$$
 Hence, since by induction
hypothesis we have
$$\Wsdn^{k}(\matr{X}_{_{1}},\gamma|_{_{{\matr{X}_{_{1}}}}},{\cal
R}_{_{T}})=t+2$$ and
$\Wsdn^{k}(\matr{X}_{_{2}},\gamma|_{_{{\matr{X}_{_{2}}}}},{\cal
R}_{_{T}})$ is at least equal to $1$, by Corollary~\ref{CCK_COL} we
have
$$\Wsdn^{k}(\matr{G}_{_{t+1}}[\rho],\gamma,{\cal R}_{_{T}}) > \max\{t+2,1\}=t+2.$$
On the other hand, by taking $B_{_{2}} \isdef B_{_{1}} \cup
\{x^{^{t+1}}_{_{1}}\}$, the list coloring problem \linebreak
$(\matr{G}_{_{t+1}}[\rho],{\cal L}^{^ {B_{_{2}},\gamma}},3)$ is
$({\cal R}_{_{T}},2,k+1)$-solvable,
 which shows that
$$\Wsdn^{k}(\matr{G}_{_{t+1}}[\rho],\gamma,{\cal R}_{_{T}})=t+3,$$
and the induction is complete.\\
For the second part, first we prove the following,
\begin{itemize}
\item{{\bf Claim 1.} Let $A \subseteq V(\matr{G}_{_{\xi}}[\rho,\gamma])$ be such that  the list coloring problem
$(\matr{G}_{_{\xi}}[\rho],{\cal L}^{^ {A,\gamma}},3)$ is $({\cal R}_{_{T}},2,k+1)$-solvable, then $|A| \geq \xi+2$.\\
We prove this by induction on $\xi$. For $\xi=1$ it is easy to see that
$$CC^{k}(\{v\},\matr{D}_{_{k+1}}[\rho,\gamma_{_{1}}],{\cal R}_{_{T}})=\{v\}.$$
Also, note that for every $x \in V(\matr{D}_{_{k+1}})$, if  $A=\{v,x\}$ then
$$CC^{k}(\{v,x\},\matr{D}_{_{k+1}}[\rho,\gamma_{_{1}}],{\cal R}_{_{T}})\subsetneq \matr{D}_{_{k+1}}[\rho,\gamma_{_{1}}], $$
while for $A=\{u,v,x_{_{1}}\}$ we have,
$$CC^{1}(A,\matr{D}_{_{k+1}}[\rho,\gamma_{_{1}}],{\cal R}_{_{T}})= \matr{D}_{_{k+1}}[\rho,\gamma_{_{1}}], $$
which shows that $|A| \geq 3$.\\
Hence, let the claim be true for any $\xi \leq s$. For $\xi=s+1$ let $\matr{G}_{_{s+1}}$ be partitioned into $\matr{X}_{_{1}}$ a copy of
$\matr{G}_{_{s}}$, and $\matr{X}_{_{2}} \isdef \matr{G}_{_{s+1}}-\matr{G}_{_{s}}$. Now, by the induction hypothesis and
 Corollary~\ref{CCK_COL}, any $\sds^{k}(\matr{X}_{_{1}},\gamma|_{_{\matr{X_{_{1}}}}},{\cal R}_{_{T}})$ is of size at least $s+2$. Since any such set for $\matr{X}_{_{2}}$
has at least one vertex, and one can verify that the hypothesis of Corollary~\ref{CCK_COL} is satisfied, we have
$$|A| \geq \Wsdn^{k}(\matr{G}_{_{s+1}},\gamma,{\cal R}_{_{T}})  \geq s+3=\xi+2.$$
}
\end{itemize}
As a direct consequence of this claim we have
$$\Ssdn^{k}(\matr{G}_{_{\xi+k}}[\rho],\gamma,{\cal R}_{_{T}}) \geq \xi+k+2,$$
and the fact that for $A \isdef
\{u,v,x^{^{1}}_{_{1}},\cdots,x^{^{\xi+k}}_{_{1}}\}$ the list
coloring problem \linebreak $(\matr{G}_{_{\xi+k}}[\rho],{\cal L}^{^
{A,\gamma}},3)$ is $({\cal R}_{_{T}},2)$-solvable, in $k$ rounds,
shows that equality holds and we actually have,
$$\Ssdn^{k}(\matr{G}_{_{\xi+k}}[\rho],\gamma,{\cal R}_{_{T}}) = \xi+k+2.$$
Moreover, we show that
$$\Ssdn^{k+1}(\matr{G}_{_{\xi+k}}[\rho],\gamma,{\cal R}_{_{T}})=k+2.$$
For this, again, by considering the coloring properties of the
vertex $v$ and the set $A=\{u,v\}$, it is easy to see that the list
coloring problem $(\matr{G}_{_{\xi +k}}[\rho],{\cal L}^{^
{A,\gamma}},3)$ is $({\cal R}_{_{T}},2)$-solvable, in exactly $k+1$
rounds and returns the coloring $\gamma$,
 which shows that
 $$\Ssdn^{k+1}(\matr{G}_{_{\xi+k}}[\rho],\gamma,{\cal R}_{_{T}}) \leq k+2.$$
To prove equality, by contradiction, let's assume that $A_{_{1}}$ is
an arbitrary set such that the list coloring problem
$(\matr{G}_{_{\xi+k}}[\rho],{\cal L}^{^ {A_{_{1}},\gamma}},3)$ is
$({\cal R}_{_{T}},2,k+2)$-solvable, with
$$\iota_{_{\matr{G}_{_{\xi +k}}[\rho,\gamma],{\cal R}_{_{T}}}}(A_{_{1}}) < k+2.$$
\begin{itemize}
\item{{\bf Claim 2.} The list coloring problem
$(\matr{G}_{_{\xi+k}}[\rho],{\cal L}^{^
{A_{_{1}},\gamma}},3)$ is $({\cal R}_{_{T}},2,k+1)$-solvable, i.e. $rounds \leq k.$\\
For this, again note that,
$$CC^{k}(\{v\},\matr{G}_{_{\xi+k}}[\rho,\gamma],{\cal R}_{_{T}})=\{v\},$$
and as in the previous case, $|A_{_{1}}| \geq 2$ with $v \in A_{_{1}}$, which shows that
$$2+rounds-1 \leq |A_{_{1}}|+rounds-1 < k+2 \ \ \Rightarrow \ \  rounds \leq k. $$
}
\end{itemize}
Therefore, using Claim~$1$ we deduce,
$$\iota_{_{\matr{G}_{_{\xi +k}}[\rho,\gamma],{\cal R}_{_{T}}}}(A_{_{1}}) \geq |A_{_{1}}| \geq \xi+k+2 > k+2,$$
which is a contradiction.
 }\end{proof}
\section{Some complexity results
}\label{COMPLEXITY}
In this section we consider a couple of basic computational problems related to
defining sets of sequential coloring of graphs and we settle their
computational complexity in the class of ${\bf NP}$-complete
problems. It should be noted that similar results in the case of {\it greedy online coloring} of graphs have already been settled  by M.~Zaker
(see \cite{ZAK01-1,ZAK01-2,ZAK06,ZAK08-1}). In this article we have considered the more general setup of {\it sequential coloring}
of graphs which asks for new proofs, even in the case of greedy algorithms, however, in what follows we not only state and prove our
theorems in this general framework, but also we focus our proofs on the case of {\it structural} rule-bases to show that the ${\bf NP}$-completeness
results are true even in this more restricted case.\\
Also, we would like to mention, in spite of the fact that we have not been able to present concise ${\bf NP}$-completeness proofs that hold for any
given rule-base, but we conjecture that our results are true  for any such nontrivial rule-base in general. Our results in this regard, (Theorem~\ref{T3}$(b)$ and Theorem~\ref{T5}) can be considered as justifications of this conjecture, where one should note that the condition of
having the rule ${\sf R}^{^{Tu}}_{_{1}}$ in the rule-base, is just a bit stronger than the natural condition of having ${\sf R}^{^{Tu}}_{_{1}}(2)$ as a necessary condition. \\
To begin let us consider the $3$-colorability problem as follows,
\begin{prb}{\bf (3COL)}\\
{\bf Given:} A graph $\matr{G}$.\\
{\bf Query:} Is $\matr{G}$ $3$-colorable?
\end{prb}
which is used to prove the ${\bf NP}$-completeness of the following problems as it is stated in the next forthcoming  theorem.
\begin{prb}{\bf (COLWDS$_{_{{\cal R}}}^{k}$)}\\
{\bf Constant:} A $d$-bounded $t$-coloring rule-base ${\cal R}$ and an integer $k \geq 1$.\\
{\bf Given:} An ordered graph $\matr{G}[n]$, and an integer $\xi \geq 1$.\\
{\bf Query:} Is it true that there exists a $t$-coloring $\gamma$
such that $\Wsdn^{k}(\matr{G}[n],\gamma,{\cal R}) \leq \xi$?
\end{prb}
\begin{prb}{\bf (COLSDS$_{_{{\cal R}}}^{k}$)}\\
{\bf Constant:} A $d$-bounded $t$-coloring rule-base ${\cal R}$ and an integer $k \geq 1$.\\
{\bf Given:} An ordered graph $\matr{G}[n]$, and an integer $\xi \geq 1$.\\
{\bf Query:} Is it true that there exists a $t$-coloring $\gamma$
such that $\Ssdn^{k}(\matr{G}[n],\gamma,{\cal R}) \leq \xi$?
\end{prb}
\begin{thm}\label{T3}
\begin{itemize}
\item[{\rm a)}]{For every integer $k \geq 1$, both problems {\bf COLWDS$_{_{{\cal R}_{_{T}}}}^{k}$}
and {\bf COLSDS$_{_{{\cal R}_{_{T}}}}^{k}$} are ${\bf NP}$-complete.}
\item[{\rm b)}]{For every $d$-bounded $3$-coloring structural rule-base ${\cal R}$ that contains Tucker's first rule ${\sf R}^{^{Tu}}_{_{1}}$,
both problems  {\bf COLWDS$^{1}_{_{{\cal R}}}$} and {\bf COLSDS$_{_{{\cal R}}}^{1}$} are ${\bf NP}$-complete.}
\end{itemize}
\end{thm}
\begin{proof}{To prove  part $(a)$ note that  clearly both problems
are in ${\bf NP}$. We show that
{\bf 3COL} $\leq^{^{P}}_{_{m}}$ {\bf COLSDS$_{_{{\cal R}_{_{T}}}}^{k}$} and
{\bf 3COL} $\leq^{^{P}}_{_{m}}$ {\bf COLWDS$_{_{{\cal R}_{_{T}}}}^{k}$}.\\
Let $\matr{G}=(V(\matr{G}),E(\matr{G}))$ be a $3$-colorable ordered graph with $|V(\matr{G})|=n$, vertex ordering $\varsigma$ and the
edge ordering $\epsilon$.
We construct the new graph $\tilde{\matr{G}}$ by substituting every edge $uv$ of $\matr{G}$
by a copy of $\sum_{i=1}^{n}{\matr{D}_{_{k}}^{^{i}}[u,v]}$, i.e.
$$\tilde{\matr{G}}[\rho] \isdef \sum_{uv \in E(G)} \sum_{i=1}^{n}{\matr{D}_{_{k}}^{^{i,\epsilon(uv)}}[\varrho][u,v]},$$
where $\matr{D}_{_{k}}^{^{i,j}}[u,v]$'s are isomorphic copies of the
graph $\matr{D}_{_{k}}[\varrho][u,v]$ defined in
Definition~\ref{DnGraph}. Also, we define the ordering $\rho$ on
vertices of $\tilde{\matr{G}}$,
$$ \rho(z) \isdef \left \{ \begin{array}{lll}
\varsigma(z)& z \in V(\matr{G}),\\
n+(3k-1)(n(j-1)+i-1)+\varrho(z)&  \{u,v\} \not \ni  z \in V(\matr{D}^{i,j}_{_{k}}[u,v]),
\end{array} \right. $$
where $\varrho$ is as defined in Definition~{\ref{DnGraph}}. We show
that $\matr{G}$ is $3$-colorable if and only if among all
$3$-colorings of $\tilde{\matr{G}}[\rho]$, there exists a
$3$-coloring $\gamma$ such that,
$$\Wsdn^{k}(\tilde{\matr{G}}[\rho],\gamma,{\cal R}_{_{T}}) \leq n.$$
To see this, consider a proper $3$-coloring $\gamma_{_{1}}$ of $\matr{G}$ and let
$$A \isdef \displaystyle{\bigcup_{_{v \in V(\matr{G})}}} v \subseteq V(\matr{\tilde{G}}).$$
Now for every $uv \in  E(\matr{G})$, the $3$-coloring
$\gamma_{_{1}}$ of $\matr{G}$ assigns colors $\gamma_{_{1}}(u)$ and
$\gamma_{_{1}}(v)$. These colors induce a coloring
$\gamma_{_{\gamma_{_{1}}(u),\gamma_{_{1}}(v)}}$ on the graph
$\sum_{i=1}^{n}{\matr{D}_{_{k}}^{^{i,\epsilon(uv)}}[\varrho][u,v]}$.
This can easily be verified from the Definition~{\ref{DnGraph}}.
Define $\gamma$ on vertices of $\tilde{\matr{G}}[\varrho]$, to be
the accumulation of all such induced colorings:
$$ \gamma(z) \isdef \left \{ \begin{array}{lll}
\gamma_{_{1}}(z)& z \in V(\matr{G}),\\
\gamma_{_{\gamma_{_{1}}(u),\gamma_{_{1}}(v)}} & z \not \in
V(\matr{G}) \ \& \ z \in
V(\sum_{i=1}^{n}{\matr{D}_{_{k}}^{^{i,\epsilon(uv)}}[\varrho][u,v]}),
\end{array} \right. $$
It is easy to verify that the list coloring problem
$(\matr{\tilde{G}}[\rho,\gamma],{\cal L}^{^ {A,\gamma}},3)$ is
$({\cal R}_{_{T}},2)$-solvable, in exactly $k$ rounds and returns
the coloring $\gamma$, which shows that $A$ is a \linebreak
 $\sds^{k}(\matr{\tilde{G}}[\rho],\gamma,{\cal R}_{_{T}})$, and consequently,
$$\Wsdn^{k}(\tilde{\matr{G}}[\rho],\gamma,{\cal R}_{_{T}}) \leq  |A| = n.$$
On the other hand, if $\matr{G}$ is not $3$-colorable, then for every $3$-coloring $\gamma_{_{2}}$ of $\matr{G}$, there exists an edge
$xy \in E(G)$ such that $\gamma_{_{2}}(x)=\gamma_{_{2}}(y)$.
But since ${\cal R}_{_{T}}$ is a structural rule-base, it is easy to see that
$$CC^{k}(\{x,y\},\tilde{\matr{G}}[\rho,\gamma_{_{2}}],{\cal R}_{_{T}})=\tilde{\matr{G}}|_{_{\{x,y\}}} \simeq \overline{\matr{K}}_{_{2}},$$
which means that the algorithm can not color the structure between
$x$ and $y$. Now, the same kind of induction used in
Theorem~\ref{NONTRIVSPEC} can be used to show that every
$\sds_{{\cal R}_{_{T}}}$ for $\tilde{\matr{G}}[\rho,\gamma_{_{2}}]$
needs at least one vertex in each copy of $\matr{D_{_{k}}}$, to
prove that
$$\Wsdn^{k}( \tilde{\matr{G}}[\rho],\gamma_{_{2}},{\cal R}_{_{T}}) > n,$$
which is a contradiction.\\
For the case of $\Ssdn^{k}$, using the same approach as in the previous case, we define
$$\tilde{\matr{H}}[\rho'] \isdef \sum_{uv \in E(G)} \sum_{i=1}^{n+k}{\matr{D}_{_{k}}^{^{i,\epsilon(uv)}}[\varrho][u,v]},$$
with ordering $\rho'$, which is defined in the similar fashion that
$\rho$ was defined. One may prove that $\matr{G}$ is $3$-colorable
if and only if among all $3$-colorings of $\tilde{\matr{H}}[\rho']$,
there exists a $3$-coloring $\gamma$ such that,
$$\Ssdn^{k}(\tilde{\matr{H}}[\rho'],\gamma,{\cal R}_{_{T}}) \leq n+k-1,$$
and this completes the proof of part $(a)$.\\
Part $(b)$ can also be proved using the same method along with part $(b)$ of Proposition~\ref{basics},
using the graph
$$\tilde{\matr{N}} \isdef \sum_{uv \in E(G)} \sum_{i=1}^{n}{\matr{D}_{_{1}}^{^{i}}[u,v]}.$$
}\end{proof}
In what follows we focus on the following problems which are natural specialization of the previous problems
to colored graphs.
\begin{prb}{\bf (WDS$_{_{{\cal R}}}^{k}$)}\\
{\bf Constant:} A $d$-bounded $t$-coloring rule-base ${\cal R}$ and an integer $k \geq 1$.\\
{\bf Given:} An ordered colored graph $\matr{G}[n,\gamma]$, and an integer $\xi \geq 1$.\\
{\bf Query:} Is it the case that $\Wsdn^{k}(\matr{G}[n],\gamma,{\cal R}) \leq \xi$?
\end{prb}

\begin{prb}{\bf (SDS$_{_{{\cal R}}}^{k}$)}\\
{\bf Constant:} A $d$-bounded $t$-coloring rule-base ${\cal R}$ and an integer $k \geq 1$.\\
{\bf Given:} An ordered colored graph $\matr{G}[n,\gamma]$, and an integer $\xi \geq 1$.\\
{\bf Query:} Is it the case that $\Ssdn^{k}(\matr{G}[n],\gamma,{\cal R}) \leq \xi$?
\end{prb}
Given a graph $\matr{G}=(V(G),E(G))$, a (vertex) {\it covering set} $S \subseteq V(G)$ is a subset of vertices such that
every edge in $E(G)$ is incident to at least one vertex in $S$. In this regard one may consider the following problem,
\begin{prb}{\bf (VertexCover)}\\
{\bf Given:} A graph $\matr{G}$ and an integer $\xi$.\\
{\bf Query:} Does $\matr{G}$ have a covering set of size at most $\xi$?
\end{prb}
Let us also define a couple of graphs as follows.
\begin{defin}
{ \label{FnGraph} For each $k \geq 2$ we define the ordered colored
graph $\matr{F}_{_{k}}[\mu,\psi][u,v,x,y]$ as follows (see
Figure~\ref{GraphF}),
$$
\begin{array}{lll}
\matr{F}_{_{k}}[\mu,\psi][u,v,x,y]&\isdef
\matr{\varepsilon}[z_{_{1}},x]+\matr{\varepsilon}[\tilde{z}_{_{1}},y]+
\matr{\varepsilon}[z_{_{1}},v]+
\matr{\varepsilon}[\tilde{z}_{_{1}},v]+\matr{\varepsilon}[u,z_{_{3}}]\\
&\\
&+\matr{\varepsilon}[u,\tilde{z}_{_{3}}]+\matr{\varepsilon}[u,z_{_{5}}]
+\matr{\varepsilon}[u,\tilde{z}_{_{5}}]\\
&\\
&+\matr{K}_{_{3}}[z_{_1},z_{_2},z_{_3}]+\matr{K}_{_{3}}[\tilde{z}_{_1},\tilde{z}_{_2},\tilde{z}_{_3}]
+\matr{K}_{_{3}}[z_{_4},z_{_2},z_{_3}]+\matr{K}_{_{3}}[\tilde{z}_{_4},\tilde{z}_{_2},\tilde{z}_{_3}]\\
&\\
&+\matr{K}_{_{3}}[z_{_5},z_{_4},y]+\matr{K}_{_{3}}[\tilde{z}_{_5},\tilde{z}_{_4},x]\\
&\\
&+\displaystyle{\sum^{{k-1}}_{{i=1}}}\ (\matr{K}_{_{3}}[y_{_{i}},x_{_1,i},x_{_{2,i}}]+\matr{K}_{_{3}}[y_{_{i+1}},x_{_{1,i}},x_{_{2,i}}])\\
&\\
&+\displaystyle{\sum^{{k-1}}_{{i=1}}}\ (\matr{\varepsilon}[u,x_{_{1,i}}])\\
&\\
&+\matr{\varepsilon}[x,y_{_{k}}]+\matr{\varepsilon}[y,y_{_{k}}]+\matr{\varepsilon}[v,y_{_{k}}],\\
\end{array}
$$
with the ordering
$$ \mu \isdef \left \{ \begin{array}{lllllllll}
\mu(u)=1,  \mu(v)=2,  \\
\mu(x)=3 , \mu(y)=4,\\
\mu(z_{_{1}})=5, \mu(\tilde{z}_{_{1}})=6,\\
\mu(z_{_{4}})=7, \mu(\tilde{z}_{_{4}})=8,\\
\mu(z_{_{5}})=9, \mu(\tilde{z}_{_{5}})=10,\\
\mu(z_{_{3}})=11, \mu(\tilde{z}_{_{3}})=12,\\
\mu(z_{_{2}})=13, \mu(\tilde{z}_{_{2}})=14,\\
\mu(y_{_{i}})=15+(i-1) & \forall \ i,\\
\mu(x_{_{1,i}})=15+k+2(i-1) & \forall \ i,\\
\mu(x_{_{2,i}})=16+k+2(i-1) & \forall \ i,\\
\end{array} \right. $$
and the coloring (as depicted in Figure~\ref{GraphF})
$$ \psi \isdef \left \{ \begin{array}{lllllllll}
\psi(u)=1,  \psi(v)=2,  \\
\psi(x)=1 , \psi(y)=1,\\
\psi(z_{_{1}})=3, \psi(\tilde{z}_{_{1}})=3,\\
\psi(z_{_{4}})=3, \psi(\tilde{z}_{_{4}})=3,\\
\psi(z_{_{5}})=2, \psi(\tilde{z}_{_{5}})=2,\\
\psi(z_{_{3}})=2, \psi(\tilde{z}_{_{3}})=2,\\
\psi(z_{_{2}})=1, \psi(\tilde{z}_{_{2}})=1,\\
\psi(y_{_{i}})=3 & \forall \ i,\\
\psi(x_{_{1,i}})=2 & \forall \ i,\\
\psi(x_{_{2,i}})=1 & \forall \ i.\\
\end{array} \right. $$
}
\end{defin}

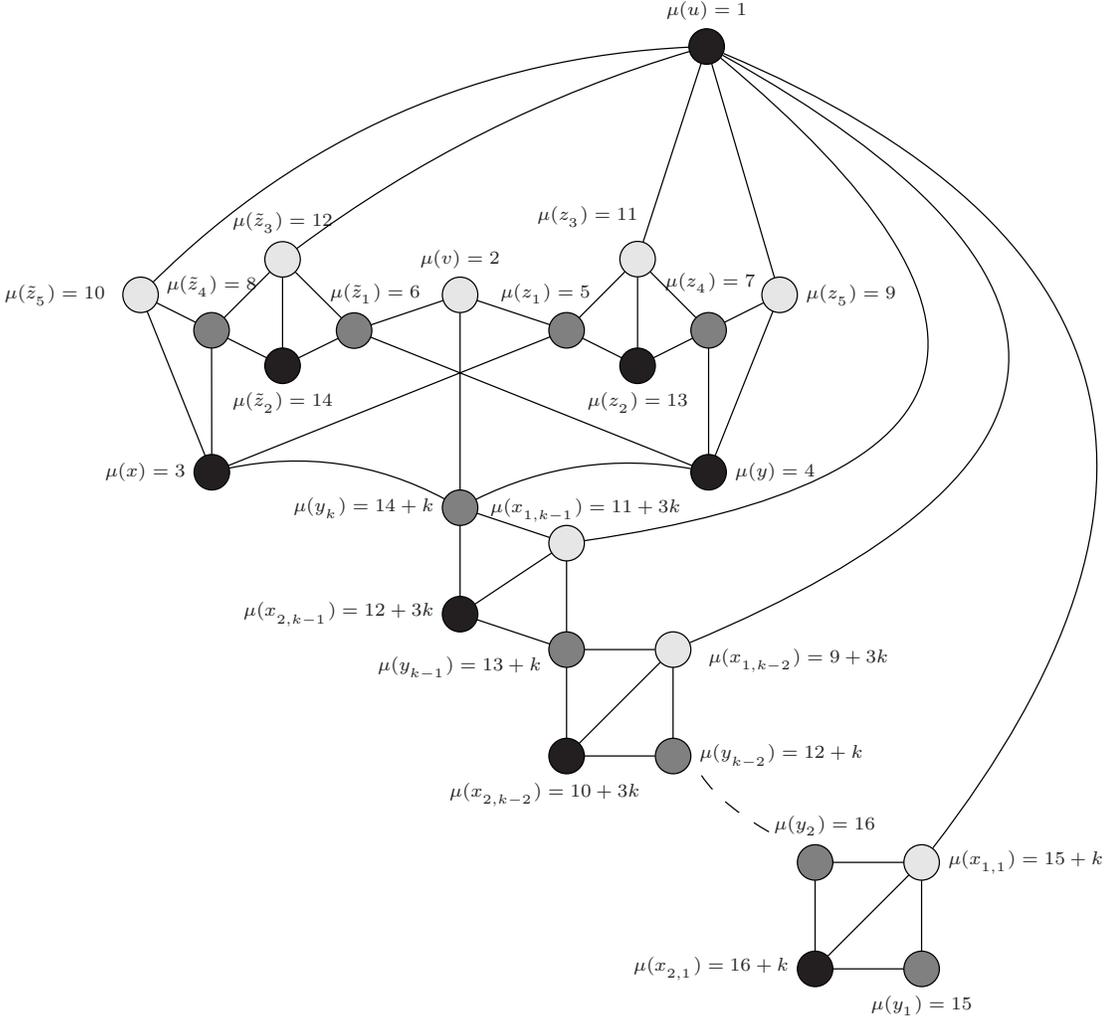
\begin{figure}[h]
\centering

\fontsize{7.5 pt}{7.5 pt}
\special{em:linewidth 1 pt} \unitlength 1.18mm \linethickness{0.5pt}
\gasset{linewidth=0.14,Nw=8.0,Nh=8.0,Nmr=4.0,Nadjustdist=1.0,ilength=5.0,flength=5.0,rdist=0.7,loopdiam=8.0,AHdist=1.41,AHLength=1.5,AHlength=1.41,ELdist=1.0}
\begin{picture}(128,122)(0,-122)
\put(0,-122){}
\gasset{Nw=4.0,Nh=4.0,Nmr=2.0,Nfill=y,ExtNL=y,NLdist=1.0,AHnb=0}
\node[fillgray=0.9](n1)(40.0,-40.0){$\mu(v)=2$}

\node[fillcolor=Black](n2)(67.78,-12.0){$\mu(u)=1$}

\node[fillcolor=Black,NLangle=180.0](n6)(12.03,-60.0){$\mu(x)=3$}

\node[fillcolor=Black,NLangle=0.0](n8)(68.0,-60.0){$\mu(y)=4$}

\node[fillgray=0.5,NLangle=60.0](n10)(28.06,-44.03){$\mu(\tilde{z}_{_{1}})=6$}

\node[fillcolor=Black,NLangle=-90.0](n13)(20.0,-48.0){$\mu(\tilde{z}_{_{2}})=14$}

\node[fillcolor=Black,NLangle=-90.0](n16)(60.0,-48.0){$\mu(z_{_{2}})=13$}

\node[fillgray=0.9](n21)(20.0,-36.0){$\mu(\tilde{z}_{_{3}})=12$}

\node[fillgray=0.9,NLdist=2.0,NLangle=180.0](n22)(4.0,-40.0){$\mu(\tilde{z}_{_{5}})=10$}

\node[fillgray=0.9,NLangle=140.0,NLdist=1.6](n25)(60.0,-35.97){$\mu(z_{_{3}})=11$}

\node[fillgray=0.9,NLangle=00.0](n26)(76.0,-40.0){$\mu(z_{_{5}})=9$}

\node[fillgray=0.5,NLdist=1.8](n28)(12.0,-44.0){$\mu(\tilde{z}_{_{4}})=8$}

\node[fillgray=0.5,NLangle=120.0](n29)(52.0,-44.0){$\mu(z_{_{1}})=5$}

\node[fillgray=0.5,NLdist=2.2,NLangle=88.0](n30)(68.0,-44.0){$\mu(z_{_{4}})=7$}

\node[fillgray=0.5,NLangle=180.0](n31)(40.0,-64.0){$\mu(y_{_{k}})=14+k$}

\node[fillgray=0.5,NLangle=190.0](n32)(52.0,-80.0){$\mu(y_{_{k-1}})=13+k$}

\node[fillgray=0.5,NLangle=0.0](n33)(64.0,-92.0){$\mu(y_{_{k-2}})=12+k$}

\node[fillgray=0.5,NLangle=75.0](n34)(80.0,-104.0){$\mu(y_{_{2}})=16$}

\node[fillgray=0.5,NLangle=-90.0](n35)(92.0,-116.0){$\mu(y_{_{1}})=15$}

\node[fillcolor=Black,NLangle=180.0](n37)(40.0,-76.0){$\mu(x_{_{2,k-1}})=12+3k$}

\node[fillcolor=Black,NLangle=-120.0](n38)(52.0,-92.0){$\mu(x_{_{2,k-2}})=10+3k$}

\node[fillcolor=Black,NLangle=180.0](n39)(80.0,-116.0){$\mu(x_{_{2,1}})=16+k$}

\node[fillgray=0.9,NLangle=60.0,NLdist=0.4](n41)(52.0,-68.0){$\mu(x_{_{1,k-1}})=11+3k$}

\node[fillgray=0.9,NLangle=-5.0,NLdist=2.0](n42)(64.0,-80.0){$\mu(x_{_{1,k-2}})=9+3k$}

\node[fillgray=0.9,NLangle=0.0](n43)(92.0,-104.0){$\mu(x_{_{1,1}})=15+k$}

\drawedge(n6,n29){}

\drawedge(n8,n10){}

\drawedge(n1,n10){}

\drawedge(n29,n1){}

\drawedge(n25,n29){}

\drawedge(n16,n29){}

\drawedge(n30,n16){}

\drawedge(n30,n25){}

\drawedge(n25,n16){}

\drawedge(n26,n30){}

\drawedge(n8,n30){}

\drawedge(n26,n8){}

\drawedge(n13,n10){}

\drawedge(n21,n10){}

\drawedge(n28,n21){}

\drawedge(n21,n13){}

\drawedge(n28,n13){}

\drawedge(n28,n6){}

\drawedge(n22,n28){}

\drawedge(n22,n6){}

\drawedge[curvedepth=2.95](n6,n31){}

\drawedge[curvedepth=-2.68](n8,n31){}

\drawedge(n1,n31){}

\drawedge(n31,n37){}

\drawedge(n41,n37){}

\drawedge(n41,n31){}

\drawedge(n32,n41){}

\drawedge(n37,n32){}

\drawedge(n38,n32){}

\drawedge(n42,n32){}

\drawedge(n33,n42){}

\drawedge(n38,n33){}

\drawedge(n42,n38){}

\drawedge(n34,n43){}

\drawedge(n35,n39){}

\drawedge(n43,n35){}

\drawedge(n34,n39){}

\drawedge(n43,n39){}

\drawedge[curvedepth=2.57](n21,n2){}

\drawedge(n25,n2){}

\drawedge[curvedepth=-7.06](n2,n22){}

\drawedge(n26,n2){}

\drawedge[curvedepth=33.62](n2,n41){}

\drawedge[curvedepth=-35.97](n42,n2){}

\drawedge[curvedepth=-31.71](n43,n2){}
\drawqbezier[dash={2.0 2.0 2.0 2.0}{0.0}](67.2,-94.19,69.5,-97.72,76.2,-101.34)

\end{picture}
\caption{\protect\label{GraphF} The graph
$\matr{F}_{_{k}}[\mu,\psi][u,v,x,y]$.  }
\end{figure}

\begin{thm}\label{T4}
Given an integer $k \geq 2$,
both problems {\bf SDS$_{_{{\cal R}_{_{T}}}}^{k}$} and {\bf WDS$_{_{{\cal R}_{_{T}}}}^{k}$}  are ${\bf NP}$-complete.
\end{thm}
\begin{proof}{
It is clear that both problems {\bf SDS$_{_{{\cal R}_{_{T}}}}^{k}$} and {\bf WDS$_{_{{\cal R}_{_{T}}}}^{k}$}
  are in ${\bf NP}$ for any $k \geq 2$. To prove the completeness we show that
 {\bf VertexCover} $\leq^{^{P}}_{_{m}}$ {\bf SDS$_{_{{\cal R}_{_{T}}}}^{k}$}. The other case {\bf VertexCover} $\leq^{^{P}}_{_{m}}$
 {\bf WDS$_{_{{\cal R}_{_{T}}}}^{k}$} can be proved using a similar argument.\\
 Let $\matr{G}=(V(\matr{G}),E(\matr{G}))$ and an integer $t$ constitute an instance of  $\matr{ \bf VertexCover}$ with $|V(\matr{G})|=n$.
We construct an ordered colored graph
$\tilde{\matr{G}}[\rho,\psi_{_{1}}]$ and an integer
$\tilde{t}=t+k+1$ that constitute an instance of {\bf SDS$_{_{{\cal
R}_{_{T}}}}^{k}$} such that $\matr{G}$ admits a vertex-cover of size
at most $t$ if and only if
$$\Ssdn^{k}(\tilde{\matr{G}}[\rho],\psi_{_{1}},{\cal R}_{_{T}}) \leq  \tilde{t}=t+k+1.$$
For this, assume the vertex ordering $\varsigma$, and an edge ordering $\epsilon$ on the vertices and edges of $\matr{G}$, respectively,
and also choose two new
vertices $u \not \in V(\matr{G})$, $v \not \in V(\matr{G})$ and define,
$$\tilde{\matr{G}} \isdef \sum_{xy \in E(\matr{G})} \sum_{i=1}^{n+k+2}{\matr{F}_{_{k}}^{^{i,\epsilon(xy)}}[\mu,\psi][u,v,x,y]},$$
 where each $\matr{F}_{_{k}}^{^{i,\epsilon(xy)}}[\mu,\psi][u,v,x,y]$ is a copy of the graph $\matr{F}_{_{k}}[\mu,\psi][u,v,x,y]$
 defined in Definition~\ref{FnGraph}.
Moreover, we define the ordering $\rho$ on $V(\tilde{\matr{G}})$ as
follows,
$$ \rho(z) \isdef \left \{ \begin{array}{lllll}
1  &  z=u,\\
2  &  z=v,\\
\varsigma(z) & z \in V(\matr{G}), \\
n+\mu(z)-2+\\
 (3k+8)[(n+k+2)(\epsilon(xy)-1)+i-1]& z \not \in \{u,v\}\ \&\\
 & z \not \in V(\matr{G}) \ \&\\
 & z \in V(\matr{F}_{_{k}}^{^{i,\epsilon(xy)}}[\mu,\psi][u,v,x,y]),
\end{array} \right. $$
as well as the coloring $\psi_{_{1}}$ of $\matr{G}[\rho]$ as,
$$ \psi_{_{1}}(z) \isdef \left \{ \begin{array}{lllll}
1 &  z=u,\\
2&  z=v,\\
1 & z \in V(\matr{G}), \\
\psi(z) & z \in
V(\matr{F}^{^{i,\epsilon(xy)}}_{_{k}}[\mu][u,v,x,y]).
\end{array} \right. $$
Now, consider a vertex covering set $A$ of $\matr{G}$ with $|A| \leq
t$. First, note that the list coloring problem
$(\matr{F}^{^i}_{_{k}}[\mu,\psi],{\cal L}^{^ {\{z,u,v\},\psi}},3)$
is $({\cal R}_{_{T}},2)$-solvable in $k$ rounds. Therefore, for
$A_{_{1}} \isdef A \cup \{u,v\}$ as a defining set, the algorithm
\begin{center}
{\bf OLG-${\cal
R}_{_{T}}$-Col}$(\matr{\tilde{G}}[\rho,\psi_{_{1}}],{\cal L}^{^
{A_{_{1}},\psi_{_{1}}}},3,2,rounds,Done)$
\end{center}
stops with $Done=true$ and $rounds=k$ resulting in coloring
$\psi_{_{1}}$, which shows that $A_{_{1}}$ is a
$\sds^{k}(\matr{\tilde{G}}[\rho],\psi_{_{1}},{\cal R}_{_{T}})$, and
consequently,
$$\Ssdn^{k}(\tilde{\matr{G}}[\rho],\psi_{_{1}},{\cal R}_{_{T}}) \leq  t+k+1.$$
On the other hand, assume that $A$ is a
$\Ssds^{k}(\tilde{\matr{G}}[\rho],\psi_{_{1}},{\cal R}_{_{T}})$ of
index less than or equal to $t+k+1$ for $\matr{\tilde{G}}$.
\begin{itemize}
\item{{\bf Claim 1.} If $(\matr{\tilde{G}}[\rho,\psi_{_{1}}],{\cal L}^{^
{A,\psi_{_{1}}}},3)$ is $({\cal R}_{_{T}},2,k)$-solvable  then the
intersection of $A$ with $V(\sum_{i=1}^{n+k+2} \matr{F}_{_{k}}^{^{i,\epsilon(xy)}}[\mu,\psi][u,v,x,y])-\{u,v,x,y\}$ for each edge $xy$, is of size at least $n+k+2$.\\
To prove this, one may use a similar induction as used in Theorem~\ref{NONTRIVSPEC} to show that
 the intersection of $A$ with each component is nonempty. }
\item{{\bf Claim 2.}  The algorithm
{\bf OLG-${\cal
R}_{_{T}}$-Col}$(\matr{\tilde{G}}[\rho,\psi_{_{1}}],{\cal L}^{^
{A,\psi_{_{1}}}},3,2,rounds,Done)$
stops in exactly $k$ rounds.
Because, if this is not the case, then by Claim~$1$,
$$\iota_{_{\matr{\tilde{G}}[\rho,\psi_{_{1}}],{\cal R}_{_{T}}}}(A) \geq |A| \geq |E(\matr{G})|(n+k+2) > t+k+1,$$
which is a contradiction.
}
\item{{\bf Claim 3.} We show that there exists another $\sds^{k}(\matr{\tilde{G}}[\rho],\psi_{_{1}},{\cal R}_{_{T}})$, say $B$, of index less than or equal $t+k+1$ such that $B \subseteq V(G) \cup \{u,v\}$.\\
For this let
$$ S_{_{xy}} \isdef A \cap \left (V(\sum_{i=1}^{n+k+2} \matr{F}_{_{k}}^{^{i,\epsilon(xy)}}[\mu,\psi][u,v,x,y])-\{u,v\} \right ). $$
It is easy to see that $|S_{_{xy}}| \leq 2$, since if $|S_{_{xy}}| > 2$ for an edge $xy$, then one may exclude all vertices of $S_{_{xy}}$ from $A$ and add $x$ and $y$ to
obtain a new $\sds^{k}(\matr{\tilde{G}}[\rho],\psi_{_{1}},{\cal R}_{_{T}})$, $A'$, with a smaller index, which is a contradiction.\\
Hence, if $|S_{_{xy}}|=2$ for an edge $xy$, we may form the new set $A' \isdef (A-S_{_{xy}}) \cup \{x,y\}$ to obtain a new
$\sds^{k}(\matr{\tilde{G}}[\rho],\psi_{_{1}},{\cal R}_{_{T}})$ with the same index.\\
Also, if $|S_{_{xy}}|=1$ for an edge $xy$, one may form the set $A' \isdef (A-S_{_{xy}}) \cup \{z\}$, where $z$ is the first vertex in $\{x,y\}$ whose color is fixed when the algorithm OLG-${\cal R}_{_{T}}$-Col is coloring the graph $\matr{\tilde{G}}$.
Note that, the structure of $\matr{F}_{_{k}}$ ensures that the colors of $x$ and $y$ can not be set in the same round of the sequential coloring algorithm using only one vertex in $S_{_{xy}}$. Also, the mapping for the case $|S_{_{xy}}|=1$ is one-to-one, since otherwise, again by replacement, one obtains a $\sds^{k}$ with a smaller  index which is a contradiction.\\
Applying this procedure for all edges of $\matr{G}$ will give rise
to a $\sds^{k}(\matr{\tilde{G}}[\rho],\psi_{_{1}},{\cal R}_{_{T}})$,
$B$, with the same index as $A$. }
\item{{\bf Claim 4.} For each edge $xy$, we have $|S_{_{xy}}| \not = 0$.\\
This is clear, since otherwise, the graph $\sum_{i=1}^{n+k+2}
\matr{F}_{_{k}}^{^{i,\epsilon(xy)}}[\mu,\psi][u,v,x,y]$ can not be
colored in $k$ rounds, i.e.
$$CC^{^{k}}(B,\matr{\tilde{G}}[\rho,\psi_{_{1}}],{\cal R}_{_{T}}) \not = \matr{\tilde{G}}[\rho,\psi_{_{1}}].$$
}
\end{itemize}
Clearly, Claim~$4$ shows that $B-\{u,v\}$ is a vertex covering set of $\matr{G}$ with
$$|B-\{u,v\}| \leq (t+k+1)-rounds+1-2=t.$$
}\end{proof}
Now we try to generalize the previous result to any structural rule-base. In this regard we define,
\begin{defin}
{ \label{AnGraph} For each $n \geq 2$ we define the ordered colored
graph $\matr{H}_{_{n}}[\omega,\eta][x,y,v]$ as follows (see
Figure~\ref{GraphA}).
\begin{equation}
\begin{array}{lll}
\matr{H}_{_{n}}[\omega,\eta][x,y,v]&\isdef
\displaystyle{\sum^{{n}}_{{i=1}}}\ (\matr{K}_{_{3}}[x,u_{_{1,i}},u_{_{2,i}}]+\matr{K}_{_{3}}[y,u_{_{1,i}},u_{_{2,i}}])\\
&\\
&\ + \ \displaystyle{\sum^{{n}}_{{i=1}}}\
\matr{\varepsilon}[v,u_{_{1,i}}],\\
&\\
\end{array}
\end{equation}
with the ordering,
$$ \omega \isdef \left \{ \begin{array}{lll}
\omega(v)=1,  \\
\omega(x)=2n+2,\\
\omega(y)=2n+3,\\
\omega(u_{_{1,i}})=2i & \forall \ i,\\
\omega(u_{_{2,i}})=2i+1 & \forall \ i,\\
\end{array} \right. $$
and the coloring (as depicted in Figure~\ref{GraphA})
$$ \eta \isdef \left \{ \begin{array}{lllllllll}
\eta(v)=3,  \\
\eta(x)=1,\\
\eta(y)=1,\\
\eta(u_{_{1,i}})=2& \forall \ i,\\
\eta(u_{_{2,i}})=3& \forall \ i.\\
\end{array} \right. $$
}
\end{defin}
\begin{figure}[t]

\centering
\fontsize{8 pt}{8 pt}
\special{em:linewidth 0.9pt} \unitlength 1.15mm \linethickness{0.5pt}
\gasset{linewidth=0.14,Nw=8.0,Nh=8.0,Nmr=4.0,Nadjustdist=1.0,ilength=5.0,flength=5.0,rdist=0.7,loopdiam=8.0,AHdist=1.41,AHLength=1.5,AHlength=1.41,ELdist=1.0}
\begin{picture}(78,94)(0,-94)
\put(0,-94){}
\gasset{Nw=4.0,Nh=4.0,Nmr=2.0,Nfill=y,fillcolor=Black,AHnb=0}
\node[ExtNL=y,NLangle=0.0,NLdist=1.0](n1)(48.0,-48.0){$\omega(y)=2n+3$}

\node[ExtNL=y,NLangle=180.0,NLdist=1.0](n2)(24.0,-48.0){$\omega(x)=2n+2$}

\node[fillgray=0.9,ExtNL=y,NLdist=1.0](n9)(36.01,-16.0){$\omega(u_{_{1,1}})=2$}

\node[fillgray=0.5,ExtNL=y,NLangle=160.0,NLdist=1.0](n30)(36.0,-28.0){$\omega(u_{_{2,1}})=3$}

\node[fillgray=0.5,ExtNL=y,NLangle=220.0,NLdist=1.0](n16)(36.0,-52.0){$\omega(u_{_{2,2}})=5$}

\node[fillgray=0.9,ExtNL=y,NLangle=90.0,NLdist=1.0](n17)(36.0,-40.16){$\omega(u_{_{1,2}})=4$}

\node[fillgray=0.5,ExtNL=y,NLangle=-90.0,NLdist=1.0](n20)(36.0,-88.0){$\omega(u_{_{2,n}})=2n+1$}

\node[fillgray=0.9,ExtNL=y,NLdist=1.0,NLangle=190.0](n21)(36.0,-76.0){$\omega(u_{_{1,n}})=2n$}

\node[fillgray=0.5,ExtNL=y,NLdist=1.0
,NLangle=00.0](n26)(90.0,-48.0){$\omega(v)=1$}

\drawedge[curvedepth=15.0](n2,n9){}

\drawedge[curvedepth=15.0](n9,n1){}

\drawedge[curvedepth=8.0](n2,n30){}

\drawedge[curvedepth=8.0](n30,n1){}

\drawedge(n1,n17){}

\drawedge(n1,n16){}

\drawedge(n17,n16){}

\drawedge(n2,n16){}

\drawedge(n2,n17){}

\drawedge(n9,n30){}

\drawedge[curvedepth=-8.8](n21,n1){}

\drawedge(n20,n21){}

\drawedge[curvedepth=-18.0](n20,n1){}

\drawedge[curvedepth=18.0](n20,n2){}

\drawedge[curvedepth=8.8](n21,n2){}

\drawedge[curvedepth=8.0](n9,n26){}

\drawedge[curvedepth=8.0](n17,n26){}

\drawedge[curvedepth=-8.0](n21,n26){}

\drawcurve[dash={2.0 2.0 2.0 3.0}{0.0}](38.0,-73.0)(42.0,-68.5)(43.0,-64.0)(42.0,-59.5)(38.0,-55)

\end{picture}
\caption{\protect\label{GraphA} The graph
$\matr{H}_{_{n}}[\omega,\eta][x,y,v]$.  }
\end{figure}
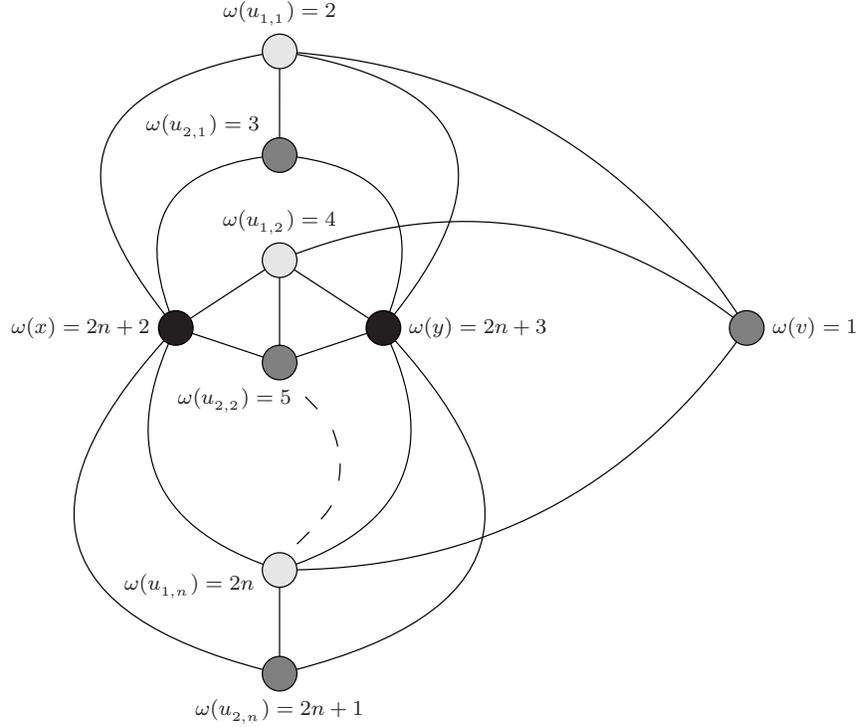

\begin{thm}\label{T5}
For every $d$-bounded $3$-coloring structural rule-base ${\cal R}$ that contains Tucker's first rule ${\sf R}^{^{Tu}}_{_{1}}$,
both problems {\bf WDS$_{_{{\cal R}}}^{1}$} and {\bf SDS$_{_{{\cal R}}}^{1}$} are ${\bf NP}$-complete.
\end{thm}
\begin{proof}{
It is clear that {\bf SDS$_{_{{\cal R}_{_{T}}}}^{1}$} and {\bf WDS$_{_{{\cal R}_{_{T}}}}^{1}$}
for $k \geq 2$  are in ${\bf NP}$. Also, using part~$(b)$ of Proposition~\ref{basics}, it is sufficient to show that
$$ {\bf VertexCover} \leq^{^{P}}_{_{m}} {\bf SDS_{_{{\cal R}}}^{1}}={\bf WDS_{_{{\cal R}}}^{1}}.$$
Let $\matr{G}=(V(\matr{G}),E(\matr{G}))$ and an integer $t$
constitute an instance of  $\matr{ \bf VertexCover}$ with
$|V(\matr{G})|=n$. We construct an ordered colored graph
$\tilde{\matr{G}}[\rho,\eta_{_{1}}]$ and an integer $\tilde{t}=t+1$
that constitute an instance of {\bf SDS$_{_{{\cal R}}}^{1}$} such
that $\matr{G}$ admits a vertex-cover of size at most $t$ if and
only if
$$\Wsdn^{1}(\tilde{\matr{G}}[\rho],\eta_{_{1}},{\cal R})) \leq  \tilde{t}=t+1.$$
Since the proof is more or less the same as the previous proof of
Theorem~\ref{T4}, we just introduce the necessary data and an sketch
of the necessary steps. For this, assume the vertex ordering
$\varsigma$, and an edge ordering $\epsilon$ on the vertices and the
edges of $\matr{G}$, respectively. Choose a new vertex $v \not \in
V(\matr{G})$ and define,
$$\tilde{\matr{G}} \isdef \sum_{_{xy \in E(\matr{G})}} \matr{H}_{_{n}}^{^{\epsilon(xy)}}[\omega,\eta][x,y,v].$$
Also, define the ordering $\rho$ on $V(\tilde{\matr{G}})$,
$$ \rho(z) \isdef \left \{ \begin{array}{llll}
1  &  z=v,\\
2n|\matr{E(G)}| +1 + \varsigma(z) & z \in V(\matr{G}), \\
1+(2n(\epsilon(xy)-1))+\omega(z) & z \not \in \{v,x,y\}\ \& \\
&  z \in V(\matr{H}_{_{n}}^{^{\epsilon(xy)}}[\omega,\eta][x,y,v]).
\end{array} \right. $$
as well as the coloring $\eta_{_{1}}$ of $\matr{G}[\rho]$ as,
$$ \eta_{_{1}}(z) \isdef \left \{ \begin{array}{llll}
3 &  z=v,\\
1 & z \in V(\matr{G}), \\
\eta(z) & z \not \in \{x,y\}\ \& \\
& z \in V(\matr{H}^{^{\epsilon(xy)}}_{_{k}}[\rho][x,y,v]).
\end{array} \right. $$
Now, consider a vertex cover $A$ of $\matr{G}$ with $|A| \leq t$, and note that
$A_{_{1}} \isdef A \cup \{v\}$ will constitute a
$\sds^{1}(\matr{\tilde{G}}[\rho],\eta_{_{1}},{\cal R}_{_{T}})$ whose size is less than or equal to $t+1$.\\
One may also prove the other side using the same method used in Theorem~\ref{T4}.
}\end{proof}
\section{Appendix: graph amalgams}\label{APPNDX}
Following  \cite{DAHT?}, and what we discussed in Section~\ref{INTRO}, note that if $\varsigma: X \longrightarrow
Y$ is a (not necessarily one-to-one) map, then one can obtain a new
marked graph $(Y,\matr{H},\tau)$ by considering the push-out of the
diagram
$$\matr{Y} \stackrel {\varsigma}{\longleftarrow} \matr{X}
\stackrel {\varrho}{\longrightarrow} \matr{G}$$ in the category of
graphs. It is easy to check that the push-out exists and is a
monomorphism. Also, it is easy to see that the new marked graph
$(Y,\matr{H},\tau)$ can be obtained from $(X,\matr{G},\varrho)$ by
identifying the vertices in each inverse-image of $\varsigma$.
Hence, again (by abuse of language) we may denote
$(Y,\matr{H},\tau)$ as
$\matr{G}[\varsigma(x_{_{1}}),\varsigma(x_{_{2}}),\ldots,\varsigma(x_{_{k}})]$
where we allow repetition in the list appearing in the brackets.
Note that with this notation one may interpret $x_{_{i}}$'s as a set
of {\it variables} in the  {\it graph structure}
$\matr{G}[x_{_{1}},x_{_{2}},\ldots,x_{_{k}}]$, such that when one
assigns other (new and not necessarily distinct) {\it values} to
these variables one can obtain
some other graphs (by identification of vertices).\\
On the other hand, given two marked graphs $(X,\matr{G},\varrho)$
and $(Y,\matr{H},\tau)$ with
$X=\{x_{_{1}},x_{_{2}},\ldots,x_{_{k}}\}$ and
$Y=\{y_{_{1}},y_{_{2}},\ldots,y_{_{l}}\}$, one can construct their
amalgam $(X,\matr{G},\varrho)+(Y,\matr{H},\tau)$ by forming the
push-out of the following diagram,
$$\matr{H}  \stackrel {\tilde{\tau}}{\longleftarrow}  \matr{X} \cap \matr{Y}
\stackrel {\tilde{\varrho}}{\longrightarrow} \matr{G},$$ in which
$\tilde{\tau} \isdef \tau|_{_{X \cap Y}}$ and $\tilde{\varrho}
\isdef  \varrho|_{_{X \cap Y}}$. Following our previous notations we
may denote the new structure by
$$\matr{G}[x_{_{1}},x_{_{2}},\ldots,x_{_{k}}]+\matr{H}[y_{_{1}},y_{_{2}},\ldots,y_{_{l}}]$$
if there is no confusion about the definition of mappings. Note that
when $\matr{X} \cap \matr{Y}$ is the empty set, then the amalgam is
the {\it disjoint union} of the two marked graphs. Also, by the
universal property of the push-out diagram, the amalgam can be
considered as marked graphs marked by $X$, $Y$, $X \cup Y$ or $X
\cap Y$.

\vspace*{.5cm}\ \\
{\bf Acknowledgement} \ \\ \ \\
Both authors wish to express their sincere thanks to M.~Zaker for
his invaluable comments for improvement. \ \\

\end{document}

%% file: macs.tex
\newcommand{\matr}[1]{\mathrm{#1}}  

\newcounter{vecdepth}
0
\newcommand{\vect}[1]{          
  \ifnum\thevecdepth=0%
    \addtocounter{vecdepth}{1}%
    \mathbf{#1}%
    \addtocounter{vecdepth}{-1}%
  \else%
    \underline{#1}%
  \fi%
}

\newcommand{\restrict}[2]{      
   {\leavevmode\lower-.2em\hbox{${{#1}_{/}}_{\! {#2}}$}}
}

\newcommand{\isdef}{            
   \setlength\unitlength{1em}
   =\put(-.975,.48){{\tiny $\triangle$}}
}

\newtheorem{remarkk}{Remark} 

\newtheorem{notee}{Note}

\newtheorem{openn}{Problem}

\newtheorem{definitionn}{Definition}

\newtheorem{examplee}{Example}

\def\probitem[#1]#2\par{\item[{\sf #1} \hspace{0ex}] #2\par}    

\newenvironment{naritemize}{%
  \begin{itemize}%
    \vspace{-\topsep}%
    \vspace{-\parskip}%
    \setlength{\itemsep}{0pt}%
    \setlength{\parskip}{0pt}%
}{%
  \end{itemize}%
}

\newcounter{seqenum}
